\documentclass[11pt,reqno]{article}
\usepackage{amsmath,amsfonts,amsthm,amssymb}
\usepackage{graphics}
\usepackage[all]{xy}
\usepackage{latexsym}
\usepackage{color}

\theoremstyle{definition}
\newtheorem{thm}{Theorem}[section]
\newtheorem{prop}[thm]{Proposition}

\newtheorem{lem}[thm]{Lemma}
\newtheorem{rem}[thm]{Remark}

\newtheorem{examp}[thm]{Example}
\newtheorem{conj}[thm]{Conjecture}

\renewcommand{\qed}{\hfill$\square$}

\begin{document}

\title{The canonical genus for Whitehead doubles of a family of alternating knots
}

\author{Hee Jeong Jang
\\{\it Department of Mathematics, Graduate School of Natural Sciences}\\
{\it Pusan National University, Busan 609-735, Korea}\\{\it E-mail: 7520jhj@hanmail.net}\\ \\and\\ \\Sang Youl Lee
\\{\it Department of Mathematics, Pusan National University,}\\
{\it Busan 609-735, Korea}\\{\it E-mail: sangyoul@pusan.ac.kr}}

\maketitle 

\begin{abstract}
For any given integer $r \geq 1$ and a quasitoric braid $\beta_r=(\sigma_r^{-\epsilon}
\sigma_{r-1}^{\epsilon}\cdots$ $ \sigma_{1}^{(-1)^{r}\epsilon})^3$ with $\epsilon=\pm 1$, we prove that the maximum degree in $z$ of the HOMFLYPT polynomial $P_{W_2(\hat\beta_r)}(v,z)$ of the doubled link $W_2(\hat\beta_r)$ of the closure $\hat\beta_r$ is equal to $6r-1$. As an application, we give a family $\mathcal K^3$ of alternating knots, including $(2,n)$ torus knots, $2$-bridge knots and alternating pretzel knots as its subfamilies, such that the minimal crossing number of any alternating knot in $\mathcal K^3$ coincides with the canonical genus of its Whitehead double. Consequently, we give a new family $\mathcal K^3$ of alternating knots for which Tripp's conjecture holds.
\end{abstract}




\section{Introduction}\label{sect-intr}

A knot is an ambient isotopy class of an oriented $1$-sphere $S^1$ smoothly embedded in the
$3$-sphere $S^3$ with a fixed standard orientation, otherwise specified. Satellite construction is one of frequently used machineries to obtain
a new knot from an arbitrary given knot.
One of famous families of satellite knots is that of $m$-twisted positive  Whitehead doubles $W_+(K,m)$
and negative Whitehead doubles $W_-(K,m)$ ($m \in \mathbb Z$), which are the satellites of knots $K$ with positive Whitehead clasp $W_+$ and negative Whitehead clasp $W_-$ as patterns, respectively (see Section 2). 

A remarkable feature of Whitehead doubles is well known facts
that the Alexander polynomial and the signature invariant of the
$0$-twisted Whitehead double of an arbitrary given knot are identical to those of the trivial knot.
Also, they have the genus one and have the unknotting number one.
In fact, Whitehead doubles are characterized as follows:
A non-trivial knot is a Whitehead double of a knot if and only if
its minimal genus and unknotting number are both $1$ \cite{ST}.

In 2002, Tripp \cite{Tri} showed that the canonical genus of a Whitehead double of a torus knot $T(2,n)$ of type $(2,n)$ is equal to $n$, the minimal crossing number of $T(2,n)$, and conjectured that the minimal crossing number of any knot coincides with the canonical genus of its Whitehead double. 
In \cite{Nak}, Nakamura has extended the tripp's argument to show that for $2$-bridge knots, Tripp's conjecture holds. He also found a non-alternating knot of which the minimal crossing number is not equal to the canonical genus of its Whitehead double and so he modified the Tripp's conjecture to the following:

\begin{conj}\label{Nakam-conj-0}
The minimal crossing number of any alternating knot coincides with the canonical genus of its Whitehead double. 
\end{conj}

In \cite{BJ}, Brittenham and Jensen showed that Conjecture \ref{Nakam-conj-0} holds for alternating pretzel knots $P(k_1,\ldots,k_n), k_1,\ldots, k_n \geq 1$ \cite[Theorem 1]{BJ}. To prove this, they used Morton's inequality \cite{Mot} and provided a method for building new knots $K$ satisfying $\max\deg_z P_{W_\pm(K,m)}(v,z)=2c(K)$ from old ones $K'$ (For more details, see Section \ref{sec-micg} or \cite{BJ}).
Actually, Brittenham and Jensen gave a larger class of alternating knots than the class including $(2,n)$-torus knots, $2$-bridge knots, and alternating pretzel knots.
In addition, Gruber \cite{Gr1} extended Nakamura's result to algebraic alternating knots in Conway's sense in a different way. 

The main purpose of this paper is to give a new infinite family of alternating knots for which Conjecture \ref{Nakam-conj-0} holds, which is an extension of the previous results of Tripp \cite{Tri}, Nakamura \cite{Nak} and Brittenham-Jensen \cite{BJ}.

This paper is organized as follows.
In Section 2, we review Whitehead double of a knot and some known preliminary results for the canonical genus of Whitehead double of a knot. In Section 3, we review the Morton's inequality for the maximum degree in $z$ of the HOMFLYPT polynomial $P_L(v,z)$ of a link $L$ and its  relation to the canonical genus of Whitehead double of a knot. We also give a brief review of Brittenham and Jensen's method.
In Section 4, we prove that for all integer $r \geq 1$, the maximum degree in $z$ of the HOMFLYPT polynomial $P_{W_2(\hat\beta_r)}(v,z)$ of the doubled link $W_2(\hat\beta_r)$ for the closure $\hat\beta_r$ of a quasitoric braid $\beta_r=(\sigma_r^{-\epsilon}
\sigma_{r-1}^{\epsilon}\cdots \sigma_{1}^{(-1)^{r}\epsilon})^3$ with $\epsilon=\pm 1$ is equal to $6r-1$ (Theorem \ref{main-thm-1}). In Section 5, we give a family $\mathcal K^3=\bigcup_{r=1}^\infty\mathcal K_r$ of alternating knots, where $\mathcal K_1$ contains all $(2,n)$ torus knots, $2$-bridge knots and alternating pretzel knots and $\mathcal K_i \not=\mathcal K_j$ if $i\not= j$, and show that the minimal crossing number of any alternating knot in $\mathcal K^3$ coincides with the canonical genus of its Whitehead double (Theorem \ref{main-thm-2}). Consequently, we give a new infinite family of alternating knots for which Conjecture \ref{Nakam-conj-0} holds. The final section 6 is devoted to prove a key lemma \ref{main-lem-1}, which has an essential role to prove Theorem \ref{main-thm-1}. 


\section{Canonical genus and Whitehead double of a knot}\label{sect-wd}

Let $T$ be a knot embedded in the unknotted solid torus $V=S^1 \times D^2$, which is essential in the sense that it meets every meridional disc in $V$. Let $K$ be an arbitrary given knot in $S^3$
and let $N(K)$ be a tubular neighborhood of $K$ in $S^3$.
Suppose $h : V=S^1 \times D^2 \to N(K)$ is a homeomorphism.
Then the image $h(T)=S_T(K)$ is a new knot, which is called a
{\it satellite (knot)} with {\it companion} $K$ and {\it pattern} $T$.
Note that if $K$ is a non-trivial knot, then satellite $S_T(K)$ is
also a non-trivial knot \cite{BZ}.

Now let $W_+$, $W_-$ and $U$ denote the positive Whitehead-clasp,
negative Whitehead-clasp and the doubled link embedded in $V$
with orientations as shown in Figure~\ref{Whitehead-clasp}.
\begin{figure}[ht]
\begin{center}
\resizebox{0.70\textwidth}{!}{%
  \includegraphics{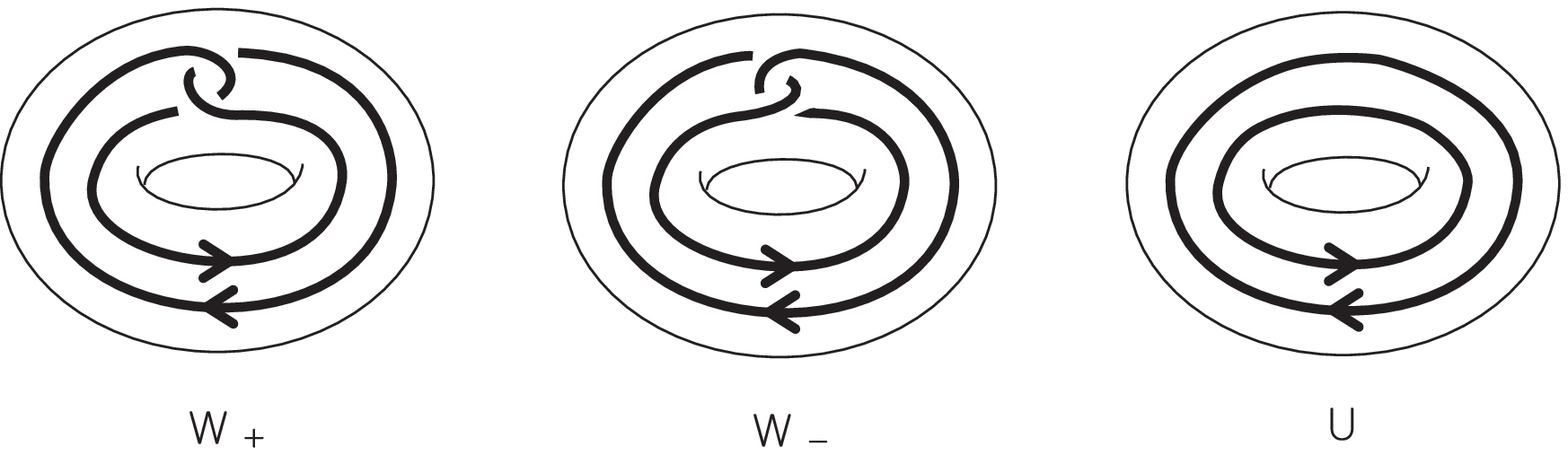}}
\caption{}\label{Whitehead-clasp}
\end{center}
\end{figure}
Let $K$ be an oriented knot and let $h : V=S^1 \times D^2 \to N(K)$ be an orientation preserving homeomorphism which take
the disk $\{\mathbf{1}\} \times D^2$ to a meridian disk of $N(K)$,
and the core $S^1 \times \{\mathbf{0}\}$ of $V$ onto the knot $K$.
Let $\ell$ be the preferred longitude of $V$. We choose an orientation for the image $h(\ell)$ so that it is parallel to $K$.
If the linking number of the image $h(\ell)$ and $K$ is equal to $m$,
then the satellite $S_{W_+}(K)$ (resp. $S_{W_-}(K)$)
with companion $K$ and pattern $W_+$ (resp. $W_-$) is called the {\it $m$-twisted positive} (resp. {\it negative}) {\it Whitehead
double} of $K$, denoted by $W_+(K, m)$(resp. $W_-(K, m)$), and the satellite $S_U(K)$ with companion $K$ and pattern $U$ is called
the {\it $m$-twisted doubled link} of $K$, denoted by $W_2(K,m)$.
The $0$-twisted positive (resp. negative) Whitehead double of $K$ is
sometimes called the {\it untwisted} positive (resp. negative)
Whitehead double of $K$.
In what follows, we use the notation $W_\pm(K, m)$ to refer the
$m$-twisted positive/negative Whitehead double of $K$ according as $+$/$-$.

Let $D$ be an oriented diagram of an oriented knot $K$ and let $w(D)$
denote the writhe of $D$, that is, the sum of the signs of all
crossings in $D$ defined by
$\mathrm{sign} \left( \xy (5,2.5);(0,-2.5) **@{-} ?<*\dir{<},
(5,-2.5);(3,-0.5) **@{-}, (0,2.5);(2,0.5) **@{-} ?<*\dir{<},
\endxy \right) = 1$ and $\mathrm{sign} \left( \xy (0,2.5);(5,-2.5)
**@{-} ?<*\dir{<}, (0,-2.5);(2,-0.5) **@{-}, (5,2.5);(3,0.5)
**@{-} ?<*\dir{<}, \endxy \right) = -1$.
Recall that for an oriented diagram $D = D_1 \cup D_2$ of an oriented
two component link $L = K_1 \cup K_2$, the {\it linking number}
$lk(L)$ of $L$ is defined to be the half of the sum of the signs of all
crossings between $D_1$ and $D_2$. The $m$-twisted positive
(resp. negative) Whitehead double $W_+(K,m)$ (resp. $W_-(K,m)$) has the
{\it canonical diagram}, denoted by  $W_+(D,m)$
(resp. $W_-(D,m)$), associated
with $D$, which is the doubled link diagram of $D$ with $(m-w(D))$
full-twists (see Figure~\ref{fig-full-twist}) and a positive Whitehead-clasp
$W_+$ (resp. negative Whitehead-clasp $W_-$) as illustrated in
(b) and (c) of Figure~\ref{fig-ex-canon-diag-trefoil}. Also, the $m$-twisted doubled link
$W_2(K,m)$ of $K$ has the canonical diagram $W_2(D,m)$ associated with $D$,
which is the doubled link diagram of $D$ with $(m-w(D))$ full-twists
without Whitehead clasp.

In particular, the canonical diagram $W_+(D,w(D))$ (resp. $W_-(D,w(D))$) of the $w(D)$-twisted positive (resp. negative) Whitehead double $W_+(K,w(D))$ (resp. $W_-(K,w(D))$) is called the
{\it standard diagram} of Whitehead double of $K$ associated with the diagram $D$ and is denoted by simply $W_+(D)$ (resp. $W_-(D)$). Likewise, the canonical diagram $W_2(D,w(D))$ of the $w(D)$-twisted doubled link $W_2(K,w(D))$ is called the
{\it standard diagram} of the doubled link of $K$ associated with the diagram $D$ and is denoted by simply $W_2(D)$ (For example, see Figure \ref{fig-ex-canon-diag-trefoil} (d)).

\begin{figure}[ht]
\centerline{ \xy (0,10);(15,10) **@{-},
(0,15);(15,15) **@{-} ,
(20,12.5);(30,12.5) **@{-} ?>*\dir{>},
(7.5,5)*{_{}},
\endxy
\quad
\xy (0,10);(1,10) **@{-}, 
(0,15);(1,15) **@{-}, (2.5,10.5);
(4,12) **@{-}, (5,13);(6.5,14.5) **@{-}, (2.5,14.5);(6.5,10.5) **@{-},
(8.5,10.5);(10,12) **@{-}, (11,13);(12.5,14.5) **@{-},
(8.5,14.5);(12.5,10.5) **@{-}, (1,10);(2.5,10.5) **\crv{(2,10)},
(1,15);(2.5,14.5) **\crv{(2,15)}, (6.5,14.5);(8.5,14.5)
**\crv{(7.5,15.5)}, (6.5,10.5);(8.5,10.5) **\crv{(7.5,9.5)},
(14,10);(12.5,10.5) **\crv{(13,10)}, (14,15);(12.5,14.5)
**\crv{(13,15)}, (14,10);(15,10) **@{-}, (14,15);(15,15) **@{-},
(7.5,6)*{_\text{$(+1)$-full twist}}, (23.5,12)*{_{\rm or}},
\endxy
\qquad
\xy (0,10);(1,10) **@{-}, 
(0,15);(1,15) **@{-}, (2.5,14.5);
(4,13) **@{-}, (5,12);(6.5,10.5) **@{-}, (2.5,10.5);(6.5,14.5) **@{-},
(8.5,14.5);(10,13) **@{-}, (11,12);(12.5,10.5) **@{-},
(8.5,10.5);(12.5,14.5) **@{-}, (1,10);(2.5,10.5) **\crv{(2,10)},
(1,15);(2.5,14.5) **\crv{(2,15)}, (6.5,14.5);(8.5,14.5)
**\crv{(7.5,15.5)}, (6.5,10.5);(8.5,10.5) **\crv{(7.5,9.5)},
(14,10);(12.5,10.5) **\crv{(13,10)}, (14,15);(12.5,14.5)
**\crv{(13,15)}, (14,10);(15,10) **@{-}, (14,15);(15,15) **@{-},
(7.5,6)*{_\text{$(-1)$-full twist}},
\endxy}
\vspace*{0pt}\caption{}\label{fig-full-twist} 
\end{figure}
\begin{figure}[ht]
\centerline{\xy (32,28)*{(a)~D}, (55,40)*{w(D)=3},
(23,44);(27,40)
**@{-}, ?<*\dir{<}, (23,40);(24.5,41.5) **@{-},
(25.5,42.5);(27,44) **@{-},
(30,44);(34,40) **@{-}, (30,40);(31.5,41.5) **@{-},
(32.5,42.5);(34,44) **@{-},
(37,44);(41,40) **@{-},  (37,40);(38.5,41.5) **@{-},
(39.5,42.5);(41,44) **@{-},
(27,44);(30,44) **\crv{(28.5,45.5)}, (27,40);(30,40)
**\crv{(28.5,38.5)}, (34,44);(37,44) **\crv{(35.5,45.5)},
(34,40);(37,40) **\crv{(35.5,38.5)}, (23,44);(41,44)
**\crv{(21,46)&(21,50)&(43,50)&(43,46)}, (23,40);(41,40)
**\crv{(21,38)&(21,34)&(43,34)&(43,38)},
\endxy}
\vskip 0.2cm
\centerline{\xy (33,22)*{(b)~{W_+(D)=W_+(D, 3)}}, 
 (22.5,42.5);(24,44) **@{-},
(25,40);(26.5,41.5) **@{-}, (21.5,42.5);(25,39) **@{-},
(24,45);(27.5,41.5) **@{-},
(27.5,42.5);(30.5,42.5) **\crv{(29,44)}, (27.5,41.5);(30.5,41.5)
**\crv{(29,40)}, (25,45);(33,45) **\crv{(29,49)}, (25,39);(33,39)
**\crv{(29,35)},
(31.5,42.5);(33,44) **@{-}, (34,40);(35.5,41.5) **@{-},
(30.5,42.5);(34,39) **@{-}, (33,45);(36.5,41.5) **@{-},
(36.5,42.5);(39.5,42.5) **\crv{(38,44)}, (36.5,41.5);(39.5,41.5)
**\crv{(38,40)}, (34,45);(42,45) **\crv{(38,49)}, (34,39);(42,39)
**\crv{(38,35)},
(40.5,42.5);(42,44) **@{-}, (43,40);(44.5,41.5) **@{-},
(39.5,42.5);(43,39) **@{-}, (42,45);(45.5,41.5) **@{-},
(31,56);(34,52.5) **\crv{(36.2,56.2)&(35.5,53)},
(32.5,52.1);(31,52) **\crv{(32,52)}, (36,52);(33,55.5)
**\crv{(31,51.8)&(31.5,54.8)}, (34.5,55.9);(36,56)
**\crv{(35,56)},
(31,56);(21.5,42.5) **\crv{(16,56.2)&(17,46)}, (31,52);(24,45)
**\crv{(20.5,52)&(21,47.5)}, (36,56);(45.5,42.5)
**\crv{(51,56.2)&(50,46)}, (36,52);(43,45)
**\crv{(46.5,52)&(46,47.5)},
(25,28);(21.5,41.5) **\crv{(16,28.5)&(17.5,38)}, (25,32);(24,39) **\crv{(20.5,32.5)&(21.5,36.5)},
(42,28);(45.5,41.5) **\crv{(51,28.5)&(49.5,38)}, (42,32);(43,39) **\crv{(46.5,32.5)&(45.5,36.5)},
(25,28);(36,28) **@{-}, (25,32);(36,32) **@{-}, (36,28);(42,28)
**@{-}, (36,32);(42,32) **@{-},
\endxy
  \quad
\xy (33,22)*{(c)~{W_+(D, 0)}}, 
(22.5,42.5);(24,44)
**@{-}, (25,40);(26.5,41.5) **@{-}, (21.5,42.5);(25,39) **@{-},
(24,45);(27.5,41.5) **@{-},
(27.5,42.5);(30.5,42.5) **\crv{(29,44)},
(27.5,41.5);(30.5,41.5)
**\crv{(29,40)}, (25,45);(33,45) **\crv{(29,49)}, (25,39);(33,39)
**\crv{(29,35)},
(31.5,42.5);(33,44) **@{-}, (34,40);(35.5,41.5) **@{-},
(30.5,42.5);(34,39) **@{-}, (33,45);(36.5,41.5) **@{-},
(36.5,42.5);(39.5,42.5) **\crv{(38,44)},
(36.5,41.5);(39.5,41.5)
**\crv{(38,40)}, (34,45);(42,45) **\crv{(38,49)}, (34,39);(42,39)
**\crv{(38,35)},
(40.5,42.5);(42,44) **@{-}, (43,40);(44.5,41.5) **@{-},
(39.5,42.5);(43,39) **@{-}, (42,45);(45.5,41.5) **@{-},
(31,56);(34,52.5) **\crv{(36.2,56.2)&(35.5,53)},
(32.5,52.1);(31,52) **\crv{(32,52)}, (36,52);(33,55.5)**\crv{(31,51.8)&(31.5,54.8)},
(34.5,55.9);(36,56) **\crv{(35,56)},
(31,56);(21.5,42.5) **\crv{(16,56.2)&(17,46)}, (31,52);(24,45)
**\crv{(20.5,52)&(21,47.5)}, (36,56);(45.5,42.5)
**\crv{(51,56.2)&(50,46)}, (36,52);(43,45)
**\crv{(46.5,52)&(46,47.5)},
(24,28);(21.5,41.5) **\crv{(16,28.5)&(17.5,38)}, (24,32);(24,39)
**\crv{(20.5,32.5)&(21.5,36.5)}, (42,28);(45.5,41.5)
**\crv{(51,28.5)&(49.5,38)}, (42,32);(43,39)
**\crv{(46.5,32.5)&(45.5,36.5)},
(24,28);(25.5,30) **\crv{(25,28)}, (25.5,30);(27,32)
**\crv{(26,32)}, (24,32);(25.2,30.5) **\crv{(24.7,32)},
(25.8,29.5);(27,28) **\crv{(26.3,28)}, (27,28);(28.5,30)
**\crv{(28,28)}, (28.5,30);(30,32) **\crv{(29,32)},
(27,32);(28.2,30.5) **\crv{(27.7,32)}, (28.8,29.5);(30,28)
**\crv{(29.3,28)},
(30,28);(31.5,30) **\crv{(31,28)}, (31.5,30);(33,32)
**\crv{(32,32)}, (30,32);(31.2,30.5) **\crv{(30.7,32)},
(31.8,29.5);(33,28) **\crv{(32.3,28)}, (33,28);(34.5,30)
**\crv{(34,28)}, (34.5,30);(36,32) **\crv{(35,32)},
(33,32);(34.2,30.5) **\crv{(33.7,32)}, (34.8,29.5);(36,28)
**\crv{(35.3,28)},
(36,28);(37.5,30) **\crv{(37,28)}, (37.5,30);(39,32)
**\crv{(38,32)}, (36,32);(37.2,30.5) **\crv{(36.7,32)},
(37.8,29.5);(39,28) **\crv{(38.3,28)}, (39,28);(40.5,30)
**\crv{(40,28)}, (40.5,30);(42,32) **\crv{(41,32)},
(39,32);(40.2,30.5) **\crv{(39.7,32)}, (40.8,29.5);(42,28)
**\crv{(41.3,28)},
\endxy
\quad
\xy (33,22)*{(d)~{W_2(D)=W_2(D,3)}}, 
 (22.5,42.5);(24,44) **@{-},
(25,40);(26.5,41.5) **@{-}, (21.5,42.5);(25,39) **@{-},
(24,45);(27.5,41.5) **@{-},
(27.5,42.5);(30.5,42.5) **\crv{(29,44)},
(27.5,41.5);(30.5,41.5)**\crv{(29,40)},
(25,45);(33,45) **\crv{(29,49)},
(25,39);(33,39)**\crv{(29,35)},
(31.5,42.5);(33,44) **@{-}, (34,40);(35.5,41.5) **@{-},
(30.5,42.5);(34,39) **@{-}, (33,45);(36.5,41.5) **@{-},
(36.5,42.5);(39.5,42.5) **\crv{(38,44)},
(36.5,41.5);(39.5,41.5)
**\crv{(38,40)}, (34,45);(42,45) **\crv{(38,49)}, (34,39);(42,39)
**\crv{(38,35)},
(40.5,42.5);(42,44) **@{-}, (43,40);(44.5,41.5) **@{-},
(39.5,42.5);(43,39) **@{-}, (42,45);(45.5,41.5) **@{-},
(31,56);(36,56) **@{-}, (31,52);(36,52) **@{-},
(31,56);(21.5,42.5) **\crv{(16,56.2)&(17,46)},
(31,52);(24,45)**\crv{(20.5,52)&(21,47.5)},
(36,56);(45.5,42.5)**\crv{(51,56.2)&(50,46)},
(36,52);(43,45)**\crv{(46.5,52)&(46,47.5)},
(25,28);(21.5,41.5) **\crv{(16,28.5)&(17.5,38)}, (25,32);(24,39) **\crv{(20.5,32.5)&(21.5,36.5)},
(42,28);(45.5,41.5) **\crv{(51,28.5)&(49.5,38)}, (42,32);(43,39) **\crv{(46.5,32.5)&(45.5,36.5)},
(25,28);(36,28) **@{-}, (25,32);(36,32) **@{-}, (36,28);(42,28)
**@{-}, (36,32);(42,32) **@{-},
\endxy
}
 \caption{}\label{fig-ex-canon-diag-trefoil} 
\end{figure}

Frankel and Pontrjagin\cite{FP} and Seifert\cite{Sei} introduced a
method to construct a compact orientable surface having a given link as its boundary.
A {\it Seifert surface} for a link $L$ in $S^3$ is a compact,
connected, and orientable surface $\Sigma$ in $S^3$ such that the
boundary $\partial\Sigma$ of $\Sigma$ is ambient isotopic to $L,$ that is, $\partial\Sigma=L.$ The {\it genus} of an oriented link $L$, denoted by $g(L)$, is the minimum genus of any Seifert surface of $L$. The genus of an unoriented link $L$ is the minumum taken over all possible choices of orientation for $L$. For a diagram $D$ of a link $L$, it is well known
that a Seifert surface for $L$ can always be obtained from $D$ by applying Seifert's algorithm\cite{Sei}. A Seifert surface for a link
$L$ constructed via Seifert's algorithm for a diagram $D$ is called
the {\it canonical Seifert surface} associated with $D$ and denoted by $\Sigma(D)$. In what follows, we denote the genus $g(\Sigma(D))$ of the canonical Seifert surface $\Sigma(D)$ by $g_c(D)$.
Then the minimum genus over all canonical Seifert surfaces for $L$
is called the {\it canonical genus} of $L$ and denoted by $g_c(L)$, i.e., $$g_c(L)=\displaystyle{\underset{\text{$D$ a diagram of $L$}}{\rm min}~g_c(D)}.$$

Seifert\cite{Sei} showed that
\begin{equation}\label{deg-AP-gen-ineq}
\frac{1}{2}{\rm deg}\Delta_K(t) \leq g(K),
\end{equation}
 where ${\rm deg}\Delta_K(t)$ is the degree of the Alexander polynomial
$\Delta_K(t)$ of $K.$ If $K$ is a torus knot, then the equality in
(\ref{deg-AP-gen-ineq}) holds, but there are also cases where the
equality does not hold. In fact, the trivial knot is the only knot
with genus zero and there are many non trivial knots whose Alexander
polynomials are equal to $1.$  Note that Seifert's algorithm
applied to a knot or link diagram might not produce a minimal genus Seifert surface and so the following inequality holds:
\begin{equation}\label{eq-inequality-genera}
g(K) \leq g_c(K).
\end{equation} Up to now, many authors have gone into finding
knots and links for which this inequality is strict or equal, for example,
see \cite{KK,LPS,LeS3,Liv,Mor,Nak,Tri} and there in. On the other hand,
Murasugi\cite{Mur} proved that if $K$ is an alternating knot, then
the equality in (\ref{deg-AP-gen-ineq}) holds and $g(K)=g_c(K)$ in
(\ref{eq-inequality-genera}). Also we have the following:

\begin{prop}\cite{BJ,Nak,Tri}\label{prop1-cr-nbr-cg-wd}
Let $K$ be a non-trivial knot and let $D$ be an oriented diagram of $K$ with $c(D)=c(K)$, where $c(K)$ denotes the minimal crossing number of $K$. Then for any integer $m$,
\begin{itemize}
\item [(1)] $g_c(W_\pm(D,m)) = g_c(W_\pm(D,w(D)))$.
\item [(2)] $g_c(W_\pm(K,m))\leq g_c(W_\pm(D,m))=c(K).$
\end{itemize}
\end{prop}

\section{Maximum $z$-degree of HOMFLYPT polynomials}\label{sec-micg}

The {\it HOMFLYPT polynomial} $P_{L}(v,z)$ (or $P(L)$ for short) of an
oriented link $L$ in $S^3$ is defined by the following three axioms:
\begin{enumerate}
\item[(1)] $P_{L}(v,z)$ is invariant under ambient isotopy of $L$.

\item[(2)] If $O$ is the trivial knot, then $P_O(v,z)=1.$

\item[(3)] If $L_{+}$, $L_{-}$ and $L_{0}$ have diagrams $D_{+}$,
$D_{-}$ and $D_{0}$ which differ as shown in Figure
\ref{fig-skein-diag-1}, then $v^{-1} P_{L_{+}}(v,z) -
vP_{L_{-}}(v,z) = zP_{L_{0}}(v,z).$
\end{enumerate}

\begin{figure}[ht]
\vspace*{5pt} \centerline{\xy (-1,14);(11,6) **@{-} ?>*\dir{>},
(11,14);(6,10.7)  **@{-} ?<*\dir{<}, (4,9.3);(-1,6) **@{-}, 
(5,-2)*{_{D_+}},
\endxy
 \qquad\qquad
\xy (-1,6);(11,14) **@{-} ?>*\dir{>}, (11,6);(6,9.3)  **@{-}
?<*\dir{<}, (4,10.7);(-1,14) **@{-}, 
(5,-2)*{_{D_-}},
\endxy
 \qquad\qquad
\xy (-1,6);(11,6) **\crv{(5,10)} ?>*\dir{>}, (-1,14);(11,14)
**\crv{(5,10)} ?>*\dir{>}, 
(5,-2)*{_{D_0}}, \endxy
 }
\vspace*{5pt}\caption{}\label{fig-skein-diag-1}
\end{figure}

Let $L$ be an oriented link and let $D$ be its oriented diagram. Then $P_{L}(v,z)$ can be computed recursively by using a skein tree,
switching and smoothing crossings of $D$ until the terminal nodes are labeled with trivial links. Observe that
\begin{eqnarray}
P_{L_{+}}(v,z) &=& v^2P_{L_{-}}(v,z) + vzP_{L_{0}}(v,z),
\label{eq-skein-pos-1} \\
P_{L_{-}}(v,z) &=& v^{-2} P_{L_{+}}(v,z) - v^{-1}zP_{L_{0}}(v,z).
\label{eq-skein-pos-2}
\end{eqnarray}
Set $\delta = (v^{-1}-v)z^{-1}$. If $L_1 \sqcup L_2$ denotes the
disjoint union of oriented links $L_1$ and $L_2$, then $P_{L_1
\sqcup L_2}(v,z) = \delta P_{L_1}(v,z) P_{L_2}(v,z)$ \cite{Cro2, Kaw}.

For the HOMFLYPT polynomial $P_L(v,z)$ of a link $L$, we denote the maximum degree in $z$ of $P_L(v,z)$ by $\max\deg_z P_L(v,z)$ or $M(L)$ for short. Let $L_+, L_-$ and $L_0$ denote the links with the diagrams $D_+, D_-$ and $D_0$, respectively, as shown in Figure \ref{fig-skein-diag-1}. Note that the degree of the sum of two polynomials cannot exceed the larger of their two degrees and is equal to the maximum of them if the two degrees are distinct. Hence it follows from (\ref{eq-skein-pos-1}) and (\ref{eq-skein-pos-2}) that
\begin{align*}
M(L_+) &\leq {\rm max}\{M(L_-), M(L_0)+1\}\\
M(L_-) &\leq {\rm max}\{M(L_+), M(L_0)+1\},\\
M(L_0) &\leq {\rm max}\{M(L_+), M(L_-)\}-1.
\end{align*}
Here, the equality holds if the two terms in the right-hand side of the inequality are distinct.

\begin{prop}\label{prop3-cr-nbr-cg-wd}
Let $K$ be an oriented knot and let $D$ be an oriented diagram of $K$.
\begin{itemize}
\item [(1)] For any integer $m$ and $\epsilon=+$ or $-$, $$M(W_2(D,m)) \leq {\rm max}\{M(W_\epsilon(D,m)), 0\}-1.$$
In particular, if $M(W_\epsilon(K,m)) > 0$, then the equality holds, i.e.,
\begin{equation}\label{eq1-prop3-cr-nbr-cg-wd}
M(W_2(D,m))=M(W_\epsilon(D,m))-1.
\end{equation}
\item [(2)] For any integer $m$, $M(W_2(D,w(D))) \leq {\rm max}\{M(W_2(D,m)), 1\}.$

In particular, if $M(W_2(D,w(D)))\not= 1$, then the equality holds, i.e.,
\begin{equation}\label{eq2-prop3-cr-nbr-cg-wd}
M(W_2(D,w(D)))=M(W_2(D,m)).
\end{equation}
\end{itemize}
\end{prop}

\begin{proof} (1) Switching one of the two crossings in the clasp of $W_+(D,m)$, we get
\begin{align*}
&v^{-1}P_{~{\xy
(31,56);(34,52.5) **\crv{(36.2,56.2)&(35.5,53)},
(32.5,52.1);(30,52) **@{-} ?>*\dir{>},
(36,52);(33,55.5)**\crv{(31,51.8)&(31.5,54.8)},
(34.5,55.9);(37,56) **@{-} ?>*\dir{>},
\endxy}~}(v,z) -
vP_{~{\xy
(32.5,56);(31,56.1) **\crv{(32,56)},
(33,55.5);(36,56) **\crv{(34,56)},
(34,55);(34,52.5) **\crv{(35,55)&(35.5,53)},
(32.5,52.1);(30,52) **@{-} ?>*\dir{>},
(36,52);(33,55.5)**\crv{(31,51.8)&(31.5,54.8)},
(34.5,55.9);(36,56) **@{-} ?>*\dir{>},
\endxy}~}(v,z) =
zP_{~{\xy
(36,56.1);(31,56) **@{-} ?<*\dir{<},
(33,55);(34,52.5) **\crv{(34,55.5)&(35.5,53)},
(32.5,52.1);(30,52) **@{-} ?>*\dir{>},
(36,52);(33,55)**\crv{(31,51.8)&(31.5,54.8)},
\endxy}~}(v,z),\\
&v^{-1}P_{W_+(D,m)}(v,z)-vP_{\xy (0,0)
*\xycircle(2,2){-}, (0.5,1.98);(0.6,2) **@{-} ?>*\dir{>} \endxy}(v,z)=zP_{W_2(D,m)}(v,z),\\
&P_{W_2(D,m)}(v,z)=v^{-1}z^{-1}P_{W_+(D,m)}(v,z)-vz^{-1}.\end{align*}
This gives the inequality $M(W_2(D,m)) \leq {\rm max}\{M(W_+(D,m)), 0\}-1$. Similarly, we obtain the inequality $M(W_2(D,m)) \leq {\rm max}\{M(W_-(D,m)), 0\}-1$. It is obvious that the equality holds if $M(W_\pm(D,m)) > 0$.


(2) Let $K$ be a non-trivial oriented knot and let $D$ be an oriented diagram of $K$. Let $W_2(D,m)$ be the canonical diagram of the $m$-twisted doubled link $W_2(K,m)$ associated with $D$. We remind that $W_2(D,m)$ is the $2$-parallel link diagram of $D$ with $m-w(D)$ full-twists. Let $n=m-w(D)$. The proof is proceeded by induction on $|n|$.

If $n=0$, then the assertion is obvious. Assume that $|n| \geq 1$ and the assertion holds for all $k < |n|$. Switching one of the $2n$ crossings among the $n$ full-twists in $W_2(D,m)$ yields $W_2(D,w(D)+\frac{n}{|n|}(|n|-1))$ (after isotopy), while smoothing the crossing yields the unknot ${\xy (0,0) *\xycircle(2,2){-}, (0.5,1.98);(0.6,2) **@{-} ?>*\dir{>} \endxy}$, and so
\begin{align*}
&P_{W_2(D,w(D)+n-1)}(v,z)=v^2P_{W_2(D,m)}(v,z)+vzP_{\xy (0,0)
*\xycircle(2,2){-}, (0.5,1.98);(0.6,2) **@{-} ?>*\dir{>} \endxy}(v,z),~\text{if}~ n \geq 0,\\
&P_{W_2(D,w(D)+n+1)}(v,z)=v^{-2}P_{W_2(D,m)}(v,z)-v^{-1}zP_{\xy (0,0)
*\xycircle(2,2){-}, (0.5,1.98);(0.6,2) **@{-} ?>*\dir{>} \endxy}(v,z),~\text{if}~ n < 0.
\end{align*}
Since $P_{\xy (0,0)
*\xycircle(2,2){-}, (0.5,1.98);(0.6,2) **@{-} ?>*\dir{>} \endxy}(v,z)=1$, if follows that
\begin{equation}\label{eqn1-deg-max}
M(W_2(D,w(D)+\frac{n}{|n|}(|n|-1))) \leq {\rm max}\{M(W_2(D,m)), 1\},
\end{equation} where the equality holds when $M(W_2(D,m)) \not= 1.$ By induction hypothesis, it follows that \begin{equation}\label{eqn2-deg-max}M(W_2(D,w(D))) \leq {\rm max}\{M(W_2(D,w(D)+\frac{n}{|n|}(|n|-1))), 1\},\end{equation}
where the equality holds when $M(W_2(D,w(D)+\frac{n}{|n|}(|n|-1))) \not= 1.$ Combining (\ref{eqn1-deg-max}) and (\ref{eqn2-deg-max}), we obtain the assertion and complete the proof.

\end{proof}

Let $D$ be an oriented link diagram. The {\it Seifert circles} of $D$ are simple closed curves obtained from $D$ by smoothing each crossing as illustrated in Figure \ref{fig-smoothing}. We denote by $s(D)$ the number of the Seifert circles of $D$.

\begin{figure}[ht]
\vspace*{5pt} \centerline{\xy (-1,14);(11,6) **@{-} ?>*\dir{>},
(11,14);(6,10.7)  **@{-} ?<*\dir{<}, (4,9.3);(-1,6) **@{-},
\endxy
 \quad
\xy (0,0)*{,},
\endxy
 \quad
\xy (-1,6);(11,14) **@{-} ?>*\dir{>}, (11,6);(6,9.3)  **@{-}
?<*\dir{<}, (4,10.7);(-1,14) **@{-},
\endxy
 \qquad
\xy (0,10);(20,10) **@{-} ?>*\dir{>}, (0,6)*{},
(10,12)*{_\mathrm{smoothing}},
\endxy
 \qquad
\xy (-1,6);(11,6) **\crv{(5,10)} ?>*\dir{>}, (-1,14);(11,14)
**\crv{(5,10)} ?>*\dir{>},
\endxy
 }
\vspace*{5pt}\caption{}\label{fig-smoothing}
\end{figure}

\begin{thm}\cite[Theorem 2]{Mot}\label{thm-Morton-ineq}
For any oriented diagram $D$ of an oriented knot or link $L$,
\begin{equation}\label{Morton-ineq-c-gen}
\max\deg_z P_L(v,z)
\leq c(D) - s(D) + 1,
\end{equation}
where $c(D)$ is the number of crossings of the diagram $D$ and $s(D)$ is the number of the Seifert circles of $D$.
\end{thm}

We note that the equality in (\ref{Morton-ineq-c-gen}) holds for alternating links, positive links, and many other links.

Let $D$ be an oriented diagram of an oriented knot or link $L$, let $\mu$ denote the number of components of $L$. Then the Euler characteristic $\chi(\Sigma(D))$ of the canonical Seifert surface $\Sigma(D)$ associated with $D$ is given by
\begin{equation*}
\chi(\Sigma(D))=s(D)-c(D)=2-2g(\Sigma(D))-\mu.
\end{equation*}
Then it follows from (\ref{Morton-ineq-c-gen}) that for every canonical Seifert surface $\Sigma(D)$ for $L$, we have
\begin{align*}
\max\deg_z P_L(v,z) &\leq c(D) - s(D) + 1=1-\chi(\Sigma(D))=2g(\Sigma(D))+\mu-1.
\end{align*}
Therefore, for a knot $K$, we obtain
\begin{equation}\label{Morton-ineq-c-gen-1}
\frac{1}{2} \max\deg_z P_K(v,z) \leq g_c(K).
\end{equation}

\begin{prop}\label{prop2-cr-nbr-cg-wd}
Let $K$ be a knot in $S^3$ with minimal crossing number $c(K)$ and let $W_\pm(K,m)$ be the $m$-twisted positive/negative Whitehead double of $K$. If $D$ is an oriented diagram of $K$ with $c(D)=c(K)$, then
\begin{align}
\frac{1}{2} \max\deg_z P_{W_\pm(K,m)}(v,z) &\leq g_c(W_\pm(K,m))\notag\\ &\leq g_c(W_\pm(D,m))=c(K).\label{Morton-ineq-c-gen-2}
\end{align}
\end{prop}

\begin{proof}
This follows from Proposition \ref{prop1-cr-nbr-cg-wd} and the inequality (\ref{Morton-ineq-c-gen-1}) at once.

\end{proof}

In the rest of this section, we briefly review Tripp's conjecture for the canonical genus of Whitehead doubles of knots. For more details, see \cite{BJ,Nak,Tri}. In \cite{Tri}, Tripp proved that the canonical genus of an $m$-twisted Whitehead double $W_{\pm}(T(2,n), m)$ of the torus knot $T(2,n)$ is equal to its crossing number, that is, 
$g_c(W_{\pm}(T(2,n), m))=n=c(T(2,n)).$ The main part of the proof is to show that the maximum $z$-degree of HOMFLYPT polynomial of Whitehead doubles of $T(2,n)$ is equal to $2c(T(2,n))$. Then he made the following:

\begin{conj} \cite[J. J. Tripp]{Tri}\label{tripp-conj}
Let $K$ be any knot with the crossing number $c(K)$. Then for any integer $m,$ 
\begin{equation}\label{eq-tripp-conj}
g_c(W_{\pm}(K, m))=c(K).\end{equation}
\end{conj}

In \cite{Nak}, Nakamura has extended the tripp's argument to show that for $2$-bridge knot $K$, Conjecture \ref{tripp-conj} holds. He also observed that the torus knot $T(4,3)$, which is not an alternating knot, does not satisfy the equality (\ref{eq-tripp-conj}) and modified the tripp's conjecture to Conjecture \ref{Nakam-conj-0} in Section 1.
In \cite{BJ}, Brittenham and Jensen showed that Conjecture \ref{Nakam-conj-0} holds for alternating pretzel knots $P(k_1,\ldots,k_n), k_1,\ldots, k_n \geq 1$ \cite[Theorem 1]{BJ}. The main tool of the proof is the following proposition \ref{BJprop2-cr-nbr-cg-wd} that follows at once by applying Proposition \ref{BJprop4-cr-nbr-cg-wd} twice, which give a method for building new knots $K$ satisfying $\max\deg_z P_{W_\pm(K,m)}(v,z)=2c(K).$ 

\begin{prop}\cite[Proposition 2]{BJ}\label{BJprop2-cr-nbr-cg-wd}
If $K'$ is a knot satisfying $$\max\deg_z P_{W_\pm(K',m)}(v,z)=2c(K'),$$ and if for a $c(K')$-minimizing diagram $D'$ for $K'$ we replace a crossing of $D'$, thought of as a half-twist, with three half-twists as shown in Figure \ref{fig:3-half twists}, producing a knot $K$, then $$\max\deg_z P_{W_\pm(K,m)}(v,z)=2c(K),$$ and therefore $g_c(W_\pm(K,m))=c(K)$.
\end{prop}

\begin{figure}[ht]
\centerline{
\xy 
(12,17);(16,13) **@{-},
(17,12);(21,8) **@{-},
(12,8);(21,17) **@{-},
(25,12.5);(35,12.5) **@{-} ?>*\dir{>},
(0,2) *{},
\endxy
\quad\quad
\xy (10,0);(10,1) **@{-},
(15,0);(15,1) **@{-}, (14.5,2.5);
(13,4) **@{-}, (12,5);(10.5,6.5) **@{-}, (10.5,2.5);(14.5,6.5) **@{-},
(14.5,8.5);(13,10) **@{-},
(12,11);(10.5,12.5) **@{-},
(10.5,8.5);(14.5,12.5) **@{-},
(10,1);(10.5,2.5) **\crv{(10,2)},
(15,1);(14.5,2.5) **\crv{(15,2)},
(14.5,6.5);(14.5,8.5) **\crv{(15.5,7.5)},
(10.5,6.5);(10.5,8.5) **\crv{(9.5,7.5)},
(10,1);(10.5,2.5) **\crv{(10,2)},
(15,1);(14.5,2.5) **\crv{(15,2)},
(14.5,6.5);(14.5,8.5) **\crv{(15.5,7.5)},
(10.5,6.5);(10.5,8.5) **\crv{(9.5,7.5)},
(14.5,12.5);(14.5,14.5) **\crv{(15.5,13.5)},
(10.5,12.5);(10.5,14.5) **\crv{(9.5,13.5)},
(14.5,14.5);(13,16) **@{-},
(12,17);(10.5,18.5) **@{-},
(10.5,14.5);(14.5,18.5) **@{-},
(10,20);(10.5,18.5) **\crv{(10,19)},
(15,20);(14.5,18.5) **\crv{(15,19)},
(10,20);(10,21) **@{-},
(15,20);(15,21) **@{-},
(25,10) *{\text{or}},
\endxy
\quad\quad
\xy (0,10);(1,10) **@{-}, 
(0,15);(1,15) **@{-}, (2.5,14.5);
(4,13) **@{-}, (5,12);(6.5,10.5) **@{-}, (2.5,10.5);(6.5,14.5) **@{-},
(8.5,14.5);(10,13) **@{-},
(11,12);(12.5,10.5) **@{-},
(8.5,10.5);(12.5,14.5) **@{-},
(1,10);(2.5,10.5) **\crv{(2,10)},
(1,15);(2.5,14.5) **\crv{(2,15)},
(6.5,14.5);(8.5,14.5) **\crv{(7.5,15.5)},
(6.5,10.5);(8.5,10.5) **\crv{(7.5,9.5)},
(1,10);(2.5,10.5) **\crv{(2,10)},
(1,15);(2.5,14.5) **\crv{(2,15)},
(6.5,14.5);(8.5,14.5) **\crv{(7.5,15.5)},
(6.5,10.5);(8.5,10.5) **\crv{(7.5,9.5)},
(12.5,14.5);(14.5,14.5) **\crv{(13.5,15.5)},
(12.5,10.5);(14.5,10.5) **\crv{(13.5,9.5)},
(14.5,14.5);(16,13) **@{-},
(17,12);(18.5,10.5) **@{-},
(14.5,10.5);(18.5,14.5) **@{-},
(20,10);(18.5,10.5) **\crv{(19,10)},
(20,15);(18.5,14.5) **\crv{(19,15)},
(20,10);(21,10) **@{-},
(20,15);(21,15) **@{-},
(0,2) *{},
\endxy}
\vspace*{0pt}\caption{}\label{fig:3-half twists} 
\end{figure}

\begin{prop}\cite[Proposition 4]{BJ}\label{BJprop4-cr-nbr-cg-wd}
If $L'$ is a non-split link with a diagram $D'$ satisfying $c(D')=c(L')$ and $$\max\deg_z P_{W_2(D')}(v,z)=2c(D')-1,$$ and $L$ is a link having diagram $D$ obtained from $D'$ by replacing a crossing in the diagram $D'$ with a full twist (so that $c(D)=c(D')+1$), then $$\max\deg_z P_{W_2(D)}(v,z)=2c(D)-1=\max\deg_z P_{W_2(D')}(v,z)+2.$$
\end{prop}

In fact, Brittenham and Jensen proved that Conjecture \ref{Nakam-conj-0} holds for a larger class of alternating knots, including $(2,n)$-torus knots, $2$-bridge knots, and alternating pretzel knots, as in the following proposition \ref{BJprop3-cr-nbr-cg-wd}:

\begin{prop}\cite[Proposition 3]{BJ}\label{BJprop3-cr-nbr-cg-wd}
Let $\mathcal K$ be the class of knots having diagrams which can be obtained from the standard  diagram of the left- or right-handed trefoil knot $T(2,3)$, the $(2,3)$ torus knot, by repeatedly replacing a crossing, thought of as a half twist, by a full twist. Then for every $K \in \mathcal K$, $$\max\deg_z P_{W_\pm(K,m)}(v,z)=2c(K),$$ and so $g_c(W_\pm(K,m))=c(K)$.
\end{prop}

The remaining part of this paper will be devoted to enlarge the class $\mathcal K$ in Proposition \ref{BJprop3-cr-nbr-cg-wd} by applying Brittenham and Jensen's argument starting with a certain class of closed quasitoric braids.


\section{Maximum $z$-degree of HOMFLYPT polynomials for doubled links of closed quasitoric braids $T(r+1,3)$}\label{sect-miwd3-plrq}

Let $r\geq 1$ be an arbitrary given integer and let $B_{r+1}$ be the $(r+1)$-strand braid group with the standard generators $\sigma_1, \sigma_2, \ldots, \sigma_{r}$ as shown in Figure \ref{fig-generators}.
\begin{figure}[ht]
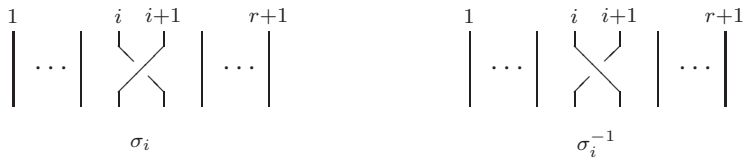

\vspace*{5pt} \centerline{ \xy (20,5);(20,7) **@{-},
(20,13);(20,15) **@{-}, (20,7);(26,13) **@{-}, (20,13);(22,11)
**@{-}, (24,9);(26,7) **@{-}, (26,5);(26,7) **@{-},
(26,13);(26,15) **@{-},
(15,5);(15,15) **@{-}, (6,5);(6,15) **@{-}, (31,5);(31,15) **@{-},
(40,5);(40,15) **@{-}, (11,10)*{\cdots}, (36,10)*{\cdots},
(6,17)*{_{1}}, (20,17)*{_{i}}, (40,17)*{_{r+1}}, (26,17)*{_{i+1}},
(23,0)*{_{\sigma_i}},
\endxy
  \qquad\qquad\qquad
\xy (20,5);(20,7) **@{-}, (20,13);(20,15) **@{-}, (20,13);(26,7)
**@{-}, (20,7);(22,9) **@{-}, (24,11);(26,13) **@{-},
(26,5);(26,7) **@{-}, (26,13);(26,15) **@{-},
(15,5);(15,15) **@{-}, (6,5);(6,15) **@{-}, (31,5);(31,15) **@{-},
(40,5);(40,15) **@{-}, (11,10)*{\cdots}, (36,10)*{\cdots},
(6,17)*{_{1}}, (20,17)*{_{i}}, (40,17)*{_{r+1}}, (26,17)*{_{i+1}},
(23,0)*{_{\sigma_i^{-1}}},
\endxy}
{}\vspace*{5pt}\caption{$\sigma_i$ and $\sigma_i^{-1}$}
\label{fig-generators}
\end{figure}

We recall that a toric braid $T(p,q)$ of type $(p,q)$ is a $p$-strand braid given by the following formula:
\begin{equation*}
T(p,q)=(\sigma_1\cdots\sigma_{p-1})^q.
\end{equation*}
The closures of toric braids yield all torus knots and links. In 2002, Manturov showed that all knots and links can be represented by the closures of a small class of braids, called quasitoric braids. We briefly review here the quasitoric braids; for more details, see \cite{Man}.

Let $m\geq 1$ and $n\geq 1$ be two integers. A braid
$\beta$ is said to be a {\it quasitoric braid} of type $(m,n)$ if it
can be expressed as an $(m+1)$-braid of the form
$$\beta=(\sigma_1^{\epsilon_{11}}\sigma_2^{\epsilon_{21}}\cdots \sigma_{m}^{\epsilon_{m1}})
(\sigma_1^{\epsilon_{12}}\sigma_2^{\epsilon_{22}}\cdots
\sigma_{m}^{\epsilon_{m2}}) \cdots
(\sigma_1^{\varepsilon_{1n}}\sigma_2^{\epsilon_{2n}}\cdots
\sigma_{m}^{\epsilon_{mn}}),$$ where $\epsilon_{ij}=\pm 1$ for all
$i=1,2,\ldots,m$ and $j=1,2,\ldots, n$. In other words, a quasitoric braid of type $(m,n)$ is a braid obtained from the standard diagram of the toric braid $T(m,n)$ by switching some crossing types. It is worth noting that the quasitoric $m$-braids form a proper
subgroup of the $m$-braid group $B_{m}$(see \cite[Proposition
1]{Man}). One of the particular utilities of the quasitoric braids is the following:

\begin{thm}\cite{Man}\label{Man-thm}
Any link can be obtained as a closure of some
quasitoric braid.
\end{thm}

In this section we consider a special class of quasitoric braids $\beta_r$ of type $(r+1,3)$ for all integers $r \geq 1$, which is a $(r+1)$-braid of the form:
\begin{equation}\label{braid-beta}
\beta_r=(\sigma_r^{\epsilon_{11}}
\sigma_{r-1}^{\epsilon_{21}}\cdots \sigma_{1}^{\epsilon_{r1}})
(\sigma_r^{\epsilon_{12}}\sigma_{r-1}^{\epsilon_{22}}\cdots
\sigma_{1}^{\epsilon_{r2}})
(\sigma_r^{\epsilon_{13}}\sigma_{r-1}^{\epsilon_{23}}\cdots
\sigma_{1}^{\epsilon_{r3}}),\end{equation} where  \begin{align}\label{braid-beta-1}
&\epsilon_{ij}=\pm 1~(1\leq i \leq r,~1\leq j \leq 3),\notag\\&\epsilon_{ij}\epsilon_{ij+1} > 0~~(1\leq i \leq r, ~1\leq j \leq 2),\\&\epsilon_{ij}\epsilon_{i+1j} < 0~~(1\leq i \leq r-1, ~1\leq j \leq 3).\notag
\end{align}

\begin{figure}[tb]
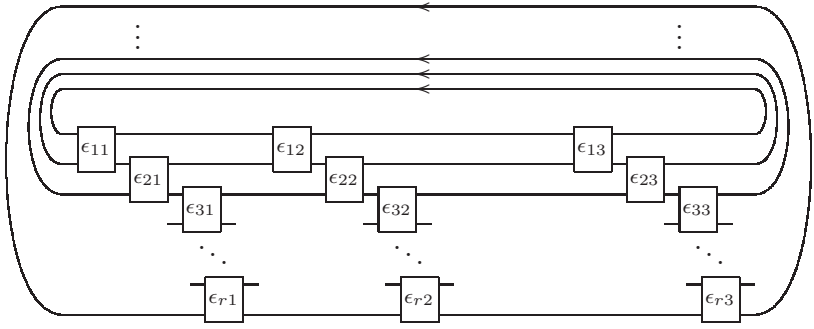

\vspace*{10pt} \centerline{ \xy
 (85,0);(90,0) **@{-}, (85,6);(90,6) **@{-},
 (85,0);(85,6) **@{-}, (90,0);(90,6) **@{-}, (87.5,3)*{_{\epsilon_{r2}}},
 (82,12);(87,12) **@{-}, (82,18);(87,18) **@{-},
 (82,12);(82,18) **@{-}, (87,12);(87,18) **@{-}, (84.5,15)*{_{\epsilon_{32}}},
 (75,16);(80,16) **@{-}, (75,22);(80,22) **@{-},
 (75,16);(75,22) **@{-}, (80,16);(80,22) **@{-}, (77.5,19)*{_{\epsilon_{22}}},
 (68,20);(73,20) **@{-}, (68,26);(73,26) **@{-},
 (68,20);(68,26) **@{-}, (73,20);(73,26) **@{-}, (70.5,23)*{_{\epsilon_{12}}},
 (66,1);(85,1) **@{-}, (90,1);(92,1) **@{-},
 (83,5);(85,5) **@{-}, (90,5);(92,5) **@{-},
 (80,13);(82,13) **@{-}, (87,13);(89,13) **@{-},
 (66,17);(75,17) **@{-}, (80,17);(82,17) **@{-}, (87,17);(92,17) **@{-},
 (66,21);(68,21) **@{-}, (73,21);(75,21) **@{-}, (80,21);(92,21) **@{-},
 (66,25);(68,25) **@{-}, (73,25);(92,25) **@{-}, (86,10)*{\ddots},
 (59,0);(64,0) **@{-}, (59,6);(64,6) **@{-},
 (59,0);(59,6) **@{-}, (64,0);(64,6) **@{-}, (61.5,3)*{_{\epsilon_{r1}}},
 (56,12);(61,12) **@{-}, (56,18);(61,18) **@{-},
 (56,12);(56,18) **@{-}, (61,12);(61,18) **@{-}, (58.5,15)*{_{\epsilon_{31}}},
 (49,16);(54,16) **@{-}, (49,22);(54,22) **@{-},
 (49,16);(49,22) **@{-}, (54,16);(54,22) **@{-}, (51.5,19)*{_{\epsilon_{21}}},
 (42,20);(47,20) **@{-}, (42,26);(47,26) **@{-},
 (42,20);(42,26) **@{-}, (47,20);(47,26) **@{-}, (44.5,23)*{_{\epsilon_{11}}},
 (40,1);(59,1) **@{-}, (64,1);(66,1) **@{-},
 (57,5);(59,5) **@{-}, (64,5);(66,5) **@{-},
 (54,13);(56,13) **@{-}, (61,13);(63,13) **@{-},
 (40,17);(49,17) **@{-}, (54,17);(56,17) **@{-}, (61,17);(66,17) **@{-},
 (40,21);(42,21) **@{-}, (47,21);(49,21) **@{-}, (54,21);(66,21) **@{-},
 (40,25);(42,25) **@{-}, (47,25);(66,25) **@{-}, (60,10)*{\ddots},
 (125,0);(130,0) **@{-}, (125,6);(130,6) **@{-},
 (125,0);(125,6) **@{-}, (130,0);(130,6) **@{-}, (127.5,3)*{_{\epsilon_{r3}}},
 (122,12);(127,12) **@{-}, (122,18);(127,18) **@{-},
 (122,12);(122,18) **@{-}, (127,12);(127,18) **@{-}, (124.5,15)*{_{\epsilon_{33}}},
 (115,16);(120,16) **@{-}, (115,22);(120,22) **@{-},
 (115,16);(115,22) **@{-}, (120,16);(120,22) **@{-}, (117.5,19)*{_{\epsilon_{23}}},
 (108,20);(113,20) **@{-}, (108,26);(113,26) **@{-},
 (108,20);(108,26) **@{-}, (113,20);(113,26) **@{-}, (110.5,23)*{_{\epsilon_{13}}},
 (106,1);(125,1) **@{-}, (130,1);(132,1) **@{-},
 (123,5);(125,5) **@{-}, (130,5);(132,5) **@{-},
 (120,13);(122,13) **@{-}, (127,13);(129,13) **@{-},
 (106,17);(115,17) **@{-}, (120,17);(122,17) **@{-}, (127,17);(132,17) **@{-},
 (106,21);(108,21) **@{-}, (113,21);(115,21) **@{-}, (120,21);(132,21) **@{-},
 (106,25);(108,25) **@{-}, (113,25);(132,25) **@{-}, (126,10)*{\ddots},
 (40,31);(132,31) **@{-}, (40,33);(132,33) **@{-}, (40,35);(132,35)
 **@{-}, (40,42);(132,42) **@{-}, (50,39)*{\vdots},
 (122,39)*{\vdots},
 (92,1);(108,1) **@{-},
 (92,17);(108,17) **@{-},
 (92,21);(108,21) **@{-},
 (92,25);(108,25) **@{-},
 (40,25);(40,31) **\crv{(38,25)&(38,31)}, (40,21);(40,33)
 **\crv{(36,21)&(36,33)}, (40,17);(40,35) **\crv{(34,17)&(34,35)},
 (40,1);(40,42) **\crv{(30,1)&(30,42)},
 (132,25);(132,31) **\crv{(134,25)&(134,31)}, (132,21);(132,33)
 **\crv{(136,21)&(136,33)}, (132,17);(132,35) **\crv{(138,17)&(138,35)},
 (132,1);(132,42) **\crv{(142,1)&(142,42)},
 (87,31);(88,31) **@{} ?<*\dir{<},
 (87,35);(88,35) **@{} ?<*\dir{<},
 (87,33);(88,33) **@{} ?<*\dir{<},
 (87,42);(88,42) **@{} ?<*\dir{<},
 \endxy
 }
\vspace*{10pt}\caption{Oriented closed braid $\hat\beta_r$}\label{fig5.1}
\end{figure}
Let $w(\beta_r)$ denote the exponent sum of $\beta_r$, i.e., 
$w(\beta_r)=
\displaystyle{\sum_{i=1}^r\sum_{j=1}^3\epsilon_{ij}}$.
Note that $w(\beta_r)$ is just the writhe of the oriented link $\hat\beta_r$, the closure of $\beta_r$.

\begin{rem}\label{prop-r-qtb}
Let $\hat\beta_r$ denote the closure of $\beta_r$ with the orientation as shown in Figure \ref{fig5.1}. Then
\begin{itemize}
\item [(1)] $\hat\beta_1$ is the right-handed trefoil knot $T(2,3)$ or the left-handed trefoil knot $T(2,3)^*$ according as $\epsilon_{11}=1$ or $\epsilon_{11}=-1$. And, $\hat\beta_2$ is the Borromean ring (see Figure \ref{fig-borromean}).
\item [(2)] $\hat\beta_r$ is a non-split alternating link without nugatory crossings and so is a minimal crossing diagram. Hence it follows that the minimal crossing number $c(\hat\beta_r)$ of $\hat\beta_r$ is given by \begin{equation}\label{cr-nbr-beta-r}
c(\hat\beta_r)=\sum_{i=1}^r\sum_{j=1}^3|\epsilon_{ij}|=3r.\end{equation}
\item [(3)] If $r=3k-1$ for some integer $k \geq 1$, then the closed braid $\hat\beta_r$ is an oriented link of three components, otherwise it is always an oriented knot.
\end{itemize}
\end{rem}

For a given oriented knot or link diagram $D$, let $W_2(D)$ denote the doubled link represented by the oriented link diagram obtained from $D$ as follows: Draw a parallel copy of $D$ pushed off of $D$ to the left according to the orientation of $D$, and then orient the parallel copy in the opposite direction. Notice that if $D$ is a knot diagram, then $W_2(D)=W_2(D,w(D))$.

Now we consider the doubled link $W_2(\hat\beta_r)$ of the closed quasitoric braid $\hat\beta_r$.
Notice that the link $W_2(\hat\beta_r)$ has no full-twists of two parallel strands and each crossing $\epsilon_{ij}$ of the closed braid diagram $\hat\beta_r$ as shown in Figure \ref{fig5.1} produces a tangle $T^{\epsilon_{ij}}_{ij}$ as shown in Figure \ref{fig5.2} in the standard diagram of $W_2(\hat\beta_r)$ associated with $\hat\beta_r$ according as $\epsilon_{ij}=1$ or $\epsilon_{ij}=-1$. The standard diagram of $W_2(\hat\beta_r)$ is equivalent to the diagram shown in Figure \ref{W2-beta-r} in which each rectangle labeled $T^{\epsilon_{ij}}_{ij}$ corresponds to the crossing $\epsilon_{ij}$ of $\hat\beta_r$.

\begin{figure}[tb]
\vspace*{10pt}
\centerline{\xy
(0,0);(-18,20) **@{-} ?>*\dir{>}, (-6,0);(-24,20) **@{-} ?<*\dir{<},
(0,20);(-8,11) **@{-} ?<*\dir{<}, (-13,5.5);(-18,0) **@{-},
(-6,20);(-11,14.5) **@{-},  (-16,9);(-24,0) **@{-}  ?>*\dir{>},
(-12.8,12.4);(-14.3,10.8) **@{-},
(-11.3,7.3);(-9.5,9.3) **@{-},
(-11.5,-5) *{\epsilon_{ij}=1},
\endxy
\qquad\qquad\qquad
\xy
(0,0);(18,20) **@{-} ?<*\dir{<}, (6,0);(24,20) **@{-} ?>*\dir{>},
(0,20);(8,11) **@{-}, (13,5.5);(18,0) **@{-} ?>*\dir{>},
(6,20);(11,14.5) **@{-} ?<*\dir{<},  (16,9);(24,0) **@{-},
(12.8,12.4);(14.3,10.8) **@{-},
(11.3,7.3);(9.5,9.3) **@{-},
(11.5,-5) *{\epsilon_{ij}=-1},
\endxy}
\vspace*{10pt}\caption{$T^{\epsilon_{ij}}_{ij}$}
\label{fig5.2}
\end{figure}
\begin{figure}[ht]
\begin{center}
\resizebox{0.80\textwidth}{!}{%
  \includegraphics{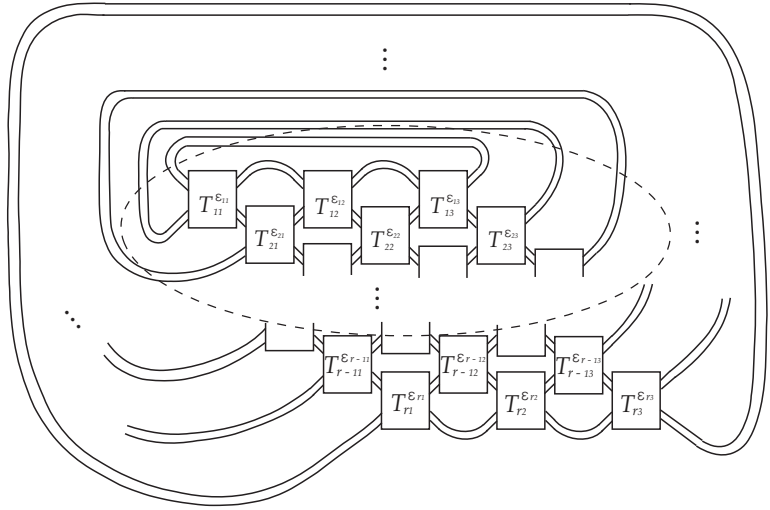}}
\caption{$W_2(\hat\beta_r)$}\label{W2-beta-r}
\end{center}
\end{figure}

In order to state the main result, we first make some notations. For our convenience, we represent the standard diagram $W_2(\hat\beta_r)$ in Figure \ref{W2-beta-r} the $r\times 3$ matrix $Q_r$ with the entries $T^{\epsilon_{ij}}_{ij}$:
$$Q_r=\left(
\begin{array}{ccc}
T^{\epsilon_{11}}_{11}&T^{\epsilon_{12}}_{12}& T^{\epsilon_{13}}_{13}\\
T^{\epsilon_{21}}_{21}&T^{\epsilon_{22}}_{22}& T^{\epsilon_{23}}_{23}\\
\vdots & \vdots & \vdots\\
T^{\epsilon_{r-11}}_{r-11}&T^{\epsilon_{r-12}}_{r-12}& T^{\epsilon_{r-13}}_{r-13}\\
T^{\epsilon_{r1}}_{r1}&T^{\epsilon_{r2}}_{r2}& T^{\epsilon_{r3}}_{r3}
\end{array}
\right).
$$
In the case that $\epsilon_{r1}=1$ (and hence $\epsilon_{r2}=\epsilon_{r3}=1$), we will denote the diagram $W_2(\hat\beta_r)$ simply by $D_r$ and let $N_r$ denote the integer given by
\begin{equation}\label{defn-N_r}
N_r=c(D_r)-s(D_r)+1=6r-1~(r \geq 1).
\end{equation}
In what follows, instead of the diagram $D_r$ illustrated in Figure \ref{W2-beta-r}, we use a shortcut diagram shown in Figure \ref{W2-beta-r-1} for $D_r$ for the sake of simplicity.
\begin{figure}[ht]
\begin{center}
\resizebox{0.65\textwidth}{!}{%
  \includegraphics{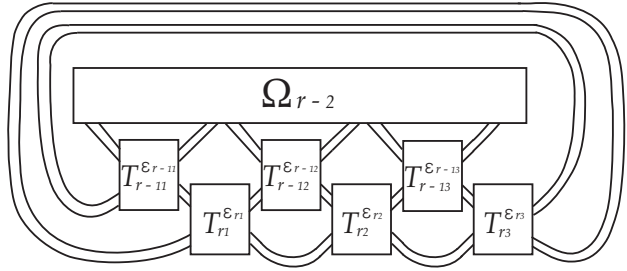}}
\caption{$D_r=W_2(\hat\beta_r)$ with $\epsilon_{r1}=1$}\label{W2-beta-r-1}
\end{center}
\end{figure}

\begin{examp}\label{examp1-br}
 Let $\beta_2$ be the quasi-toric braid of type $(3,3)$, i.e.,
\begin{equation*}
\beta_2=(\sigma_2\sigma_{1}^{-1})
(\sigma_2\sigma_1^{-1})
(\sigma_2\sigma_1^{-1}).
\end{equation*}
Then the closed braid $\hat\beta_2$ is the Borromean ring (see Figure \ref{fig-borromean}) and the $2$-parallel link $D_2=W_2(\hat\beta_2)$ is represented by $2\times 3$ matrix $Q_2$:
$$Q_{2}=\left(
\begin{array}{ccc}
T^{1}_{11}&T^{1}_{12}& T^{1}_{13}\\
T^{-1}_{21}&T^{-1}_{22}& T^{-1}_{23}\\
\end{array}
\right).
$$
\begin{figure}[ht]
\begin{center}
\resizebox{0.6\textwidth}{!}{%
  \includegraphics{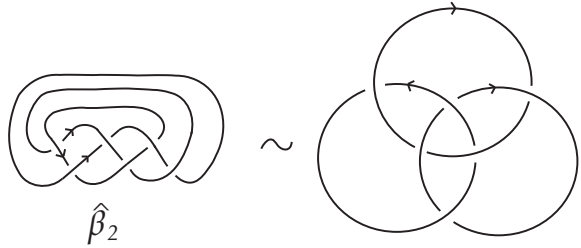}}
\caption{Borromean ring}\label{fig-borromean}
\end{center}
\end{figure}
By a direct computation, we obtain
\begin{align*}
P_{W_2(\hat\beta_2)}(v,z)
&=z^{-5}(-v^{5}+5v^{3}-10v+10v^{-1}-5v^{-3}+v^{-5})\\
&+z^{-1}(8v^{5}-40v^{3}+80v-80v^{-1}+40v^{-3}-8v^{-5})\\
&+z(12v^{5}-68v^{3}+144v-144v^{-1}+68v^{-3}+12v^{-5})\\
&+z^{3}(2v^{5}-22v^{3}+56v-56v^{-1}+22v^{-3}-2v^{-5})\\
&+z^{5}(-v^{7}-5v^{5}+13v^{3}-7v+7v^{-1}-13v^{-3}
+5v^{-5}+v^{-7})\\
&+z^{7}(-2v^{5}+8v^{3}+10v-10v^{-1}-8v^{-3}+2v^{-5})\\
&+z^{9}(v^{3}+11v-11v^{-1}-v^{-3})+z^{11}(2v-2v^{-1}).
\end{align*}
Hence the maximal $z$-degree of the HOMFLYPT polynomial $P_{W_2(\hat\beta_2)}(v,z)$ of the doubled link $W_2(\hat\beta_2)$ is given
by $$\max\deg_z P_{W_2(\hat\beta_2)}(v,z)=11=2\cdot 6-1=2c(\hat\beta_2)-1.$$
On the other hand, let $\hat\beta_2^*$ denote the mirror image of $\hat\beta_2$. Then we also have
\begin{align*}
\max\deg_z P_{W_2(\hat\beta_2^*)}(v,z)&=\max\deg_z P_{W_2(\hat\beta_2)}(v^{-1},z)\\
&=11=2\cdot 6-1=2c(\hat\beta_2^*)-1.
\end{align*}
\end{examp}

Now we construct a partial skein tree as shown in Figure \ref{resol-tree-1} for the tangle $T^{1}_{r3}$ in $D_r$ of the left hand side of Figure \ref{fig5.2}. We label all nodes in the skein tree with $A, B, E_1, F_1, F_2, F_3, F_4,$ and $G$ as shown in Figure \ref{resol-tree-1}. Now let $D_{r}^i (1\leq i \leq 8)$ denote the link diagram represented by the $r\times 3$ matrix:
$$D_{r}^i=\left(
\begin{array}{ccc}
T^{\epsilon_{11}}_{11}&T^{\epsilon_{12}}_{12}& T^{\epsilon_{13}}_{13}\\
T^{\epsilon_{21}}_{21}&T^{\epsilon_{22}}_{22}& T^{\epsilon_{23}}_{23}\\
\vdots & \vdots & \vdots\\
T^{\epsilon_{r-11}}_{r-11}&T^{\epsilon_{r-12}}_{r-12}& T^{\epsilon_{r-13}}_{r-13}\\
T^{1}_{r1}&T^{1}_{r2}& T_i
\end{array}
\right).
$$
That is, $D_{r}^i$ is the link diagram  obtained from the link diagram $D_r$ by replacing the tangle $T^{1}_{r3}$ with the tangle $T_i$, where $$T_1=A, T_2=B, T_3=E_1, T_4=F_1, T_5=F_2, T_6=F_3, T_7=F_4, T_8=G.$$ Hence two diagrams $D_r$ and $D_{r}^i$ are identical except the only one tangle corresponding to the (r,3)-entry of the matrix notations. In these terminologies, we have the following lemma \ref{main-lem-1} that will play an essential role in the proof of Theorem \ref{main-thm-1} below.

\begin{figure}[ht]
\begin{center}
\resizebox{0.32\textwidth}{!}{%
  \includegraphics{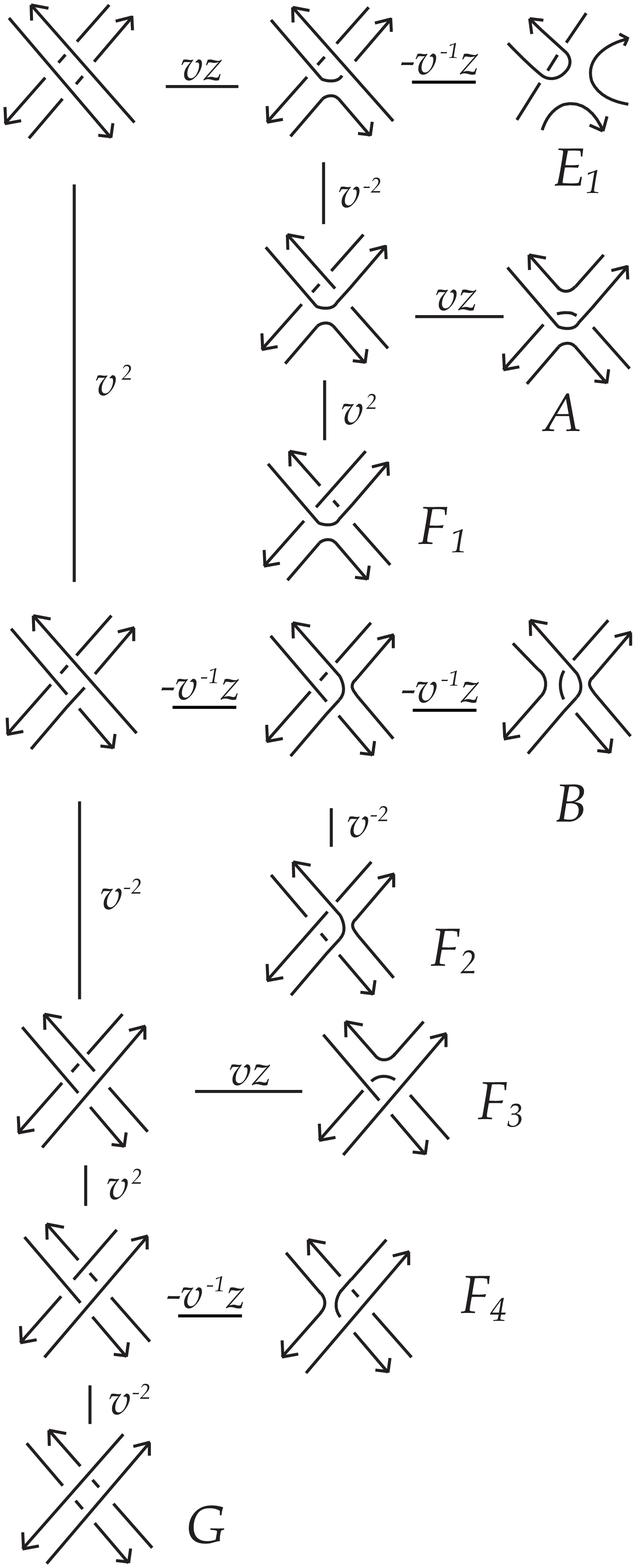}}
\caption{A partial skein tree for $T^{1}_{r3}$}\label{resol-tree-1}
\end{center}
\end{figure}

\begin{lem}\label{main-lem-1} \
\begin{itemize}
\item [(1)]
$\max\deg_z P_{D_{r}^4}(v,z)\leq N_r-3~\text{if}~ r \geq 3.$
\item [(2)] $\max\deg_z P_{D_{r}^5}(v,z) \leq N_r-3~\text{if}~ r\geq 3.
$
\item [(3)] $\max\deg_z P_{D_{r}^6}(v,z) \leq N_r-3~\text{if}~ r\geq 3.
$
\item [(4)] $\max\deg_z P_{D_{r}^7}(v,z) \leq N_r-3~\text{if}~ r\geq 3.
$
\item [(5)] $\max\deg_z P_{D_{r}^8}(v,z) \leq N_r-4~\text{if}~ r\geq 3.
$
\end{itemize}
\end{lem}
The proof of this lemma \ref{main-lem-1} will be given in the final section \ref{sect-pf-lemma}. Now, let us state our main theorem of this section.

\begin{thm}\label{main-thm-1}
Let $\beta_r (r \geq 1)$ be a quasitoric braid of type $(r+1,3)$ in (\ref{braid-beta}) and let $W_2(\hat\beta_r)$ be the doubled link of $\hat\beta_r$. Then
\begin{equation}\label{main-formula-1}
\max\deg_z P_{W_2(\hat\beta_r)}(v,z)=2c(\hat\beta_r)-1=6r-1.
\end{equation}
\end{thm}

\begin{proof} We prove the assertion (\ref{main-formula-1}) by induction on $r$. If $r=1$, then $\beta_1=\sigma_1^3$ or $\sigma_1^{-3}$, and so $\hat\beta_1$ is the right-handed trefoil knot or the left-handed trefoil knot. In either cases, it is immediate from direct calculations that $$\max\deg_z P_{W_2(\hat\beta_1)}(v,z)=\max\deg_z P_{D_1}(v,z)=5=2\cdot 3-1=2c(\hat\beta_1)-1.$$
(In the case that $r=2$, it follows from Example \ref{examp1-br} that the assertion (\ref{main-formula-1}) also holds.)

Now we assume that $r \geq 3$ and the assertion (\ref{main-formula-1}) holds for every integers $ \leq r-1$. We consider two cases separately.

\smallskip

{\bf Case I.} $\epsilon_{r3}=1$. First we observe from (\ref{braid-beta-1}) that $\epsilon_{r1}=\epsilon_{r2}=1$.
In this case, we have $W_2(\hat\beta_r)=D_r$ by the notational convention above.

\smallskip

{\bf Claim.} $\max\deg_z P_{D_r}(v,z)=2c(\hat\beta_r)-1=6r-1.$

\smallskip

{\bf Proof of Claim.} From the skein relation for the HOMFLYPT polynomial and a partial skein tree for $T^1_{r3}$ in Figure \ref{resol-tree-1}, we obtain
\begin{align}\label{eqn1-pf-main-thm-1}
P_{D_r}(v,z) &= (P_{D_{r}^1}(v,z) + P_{D_{r}^2}(v,z) - P_{D_{r}^3}(v,z))z^2 \notag\\&+ (vP_{D_{r}^4}(v,z) - v^{-1} P_{D_{r}^5}(v,z) + vP_{D_{r}^6}(v,z)-vP_{D_{r}^7}(v,z))z\notag\\& + P_{D_{r}^8}(v,z).
\end{align}
We observe that the link diagram $D^1_r$ is isotopic to the link diagram (a) of Figure \ref{fig-D-1-r}, which is isotopic to the diagram (b) in Figure \ref{fig-D-1-r}.
\begin{figure}[ht]
\begin{center}
\resizebox{0.90\textwidth}{!}{%
  \includegraphics{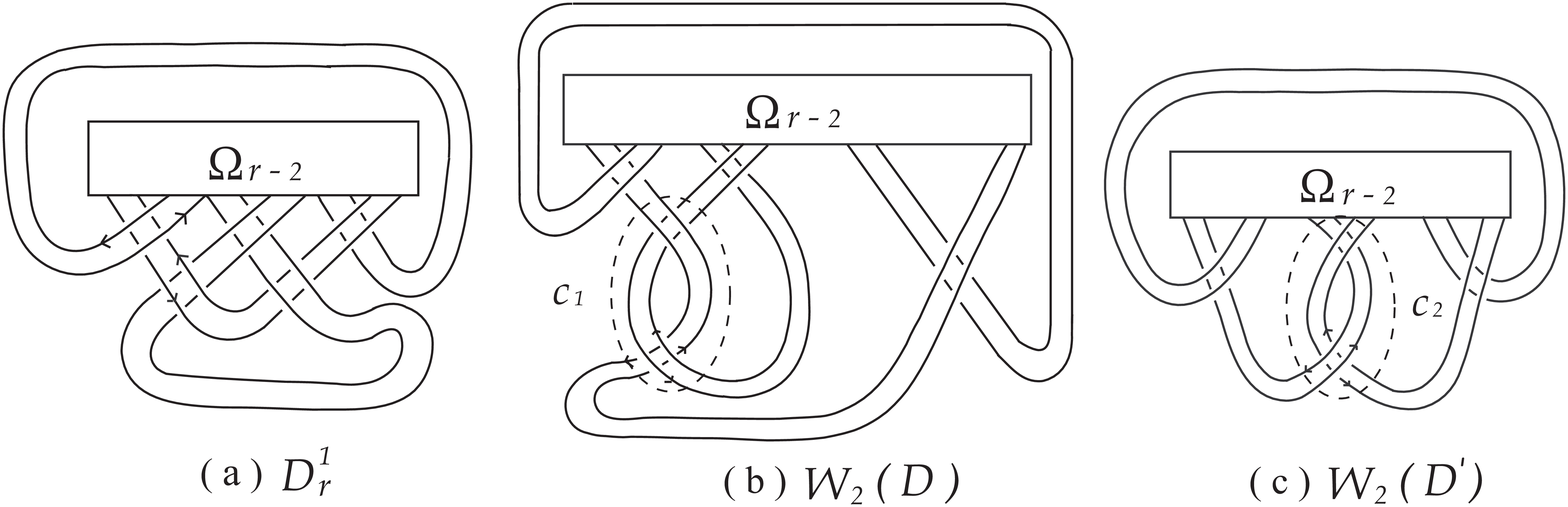}}
\caption{$D^{1}_{r}$}\label{fig-D-1-r}
\end{center}
\end{figure}

Now let $L'$ be an oriented link having diagram $D'$ obtained from the standard closed braid diagram of a non-split alternating link $\hat\beta_{r-1}$ by replacing the crossing $\sigma_1^{\epsilon_{r-12}}$ in $\hat\beta_{r-1}$ with a full twist (so that $c(D')=c(\hat\beta_{r-1})+1$) as illustrated in (a) and (b) of Figure \ref{fig-D-pf-main-thm-1}. By induction hypothesis, we have
\begin{equation}\label{main-formula-1-r-1}
\max\deg_z P_{W_2(\hat\beta_{r-1})}(v,z)=2c(\hat\beta_{r-1})-1=6(r-1)-1~ (r \geq 2).
\end{equation}
By Proposition \ref{BJprop4-cr-nbr-cg-wd}, we then obtain
\begin{align}
\max\deg_z P_{W_2(L')}(v,z)&=2c(D')-1\notag\\&=\max\deg_z P_{W_2(\hat\beta_{r-1})}(v,z)+2.\label{eqn2-pf-main-thm-1}
\end{align}
It is obvious that $L'$ is a non-split alternating link satisfying $c(L')=c(D')$ and the doubled link $W_2(L')$ has a diagram $W_2(D')$ in (c) of Figure \ref{fig-D-1-r}. Now let $L$ be an oriented link having diagram $D$ obtained from $D'$ by replacing a crossing in $D'$ with a full twist as illustrated in (c), (e) and (f) of Figure \ref{fig-D-pf-main-thm-1} so that $c(D)=c(D')+1$. Then the doubled link $W_2(L)$ has a diagram $W_2(D)$ in (b) of Figure \ref{fig-D-1-r}.
\begin{figure}[ht]
\begin{center}
\resizebox{0.70\textwidth}{!}{%
  \includegraphics{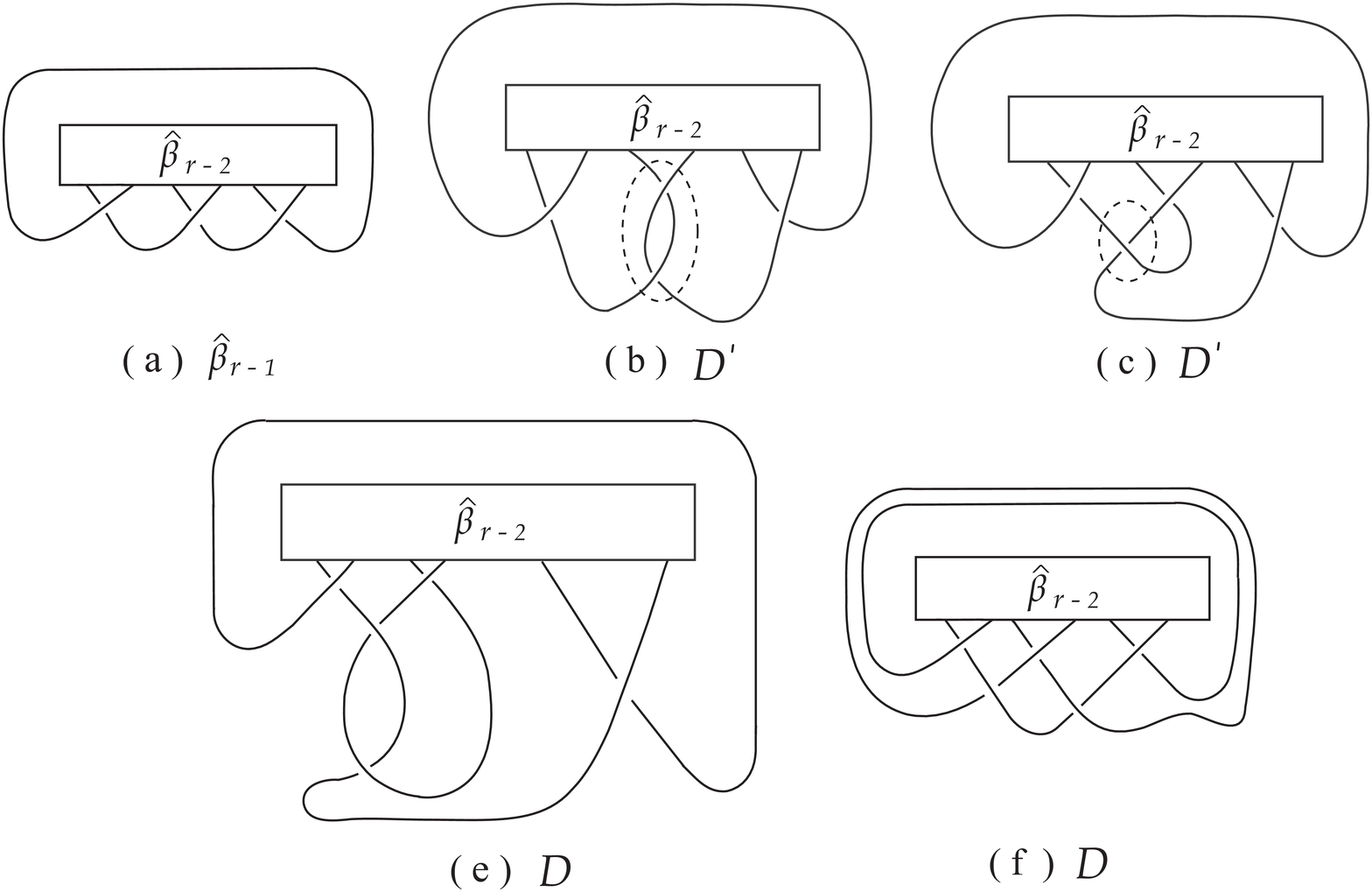}}
\caption{}\label{fig-D-pf-main-thm-1}
\end{center}
\end{figure}
By Proposition \ref{BJprop4-cr-nbr-cg-wd} again, we have
\begin{align}
\max\deg_z P_{W_2(L)}(v,z)&=2c(D)-1\notag\\
&=\max\deg_z P_{W_2(L')}(v,z)+2.\label{eqn3-pf-main-thm-1}
\end{align}
Then we obtain from (\ref{eqn2-pf-main-thm-1}) and (\ref{eqn3-pf-main-thm-1}) that
\begin{align}
\max\deg_z P_{D^1_r}(v,z) &=\max\deg_z P_{W_2(L)}(v,z)\notag\\
&=\max\deg_z P_{W_2(\hat\beta_{r-1})}(v,z)+4\notag\\
&=\max\deg_z P_{D_{r-1}}(v,z)+4.\label{eqn4-pf-main-thm-1}
\end{align}
Similarly, we observe that the link diagram $D^2_r$ is isotopic to the link diagram in the left side of Figure \ref{fig-D-2-r}, which is isotopic to the diagram in the right side of Figure \ref{fig-D-2-r}.

\begin{figure}[ht]
\begin{center}
\resizebox{0.80\textwidth}{!}{%
  \includegraphics{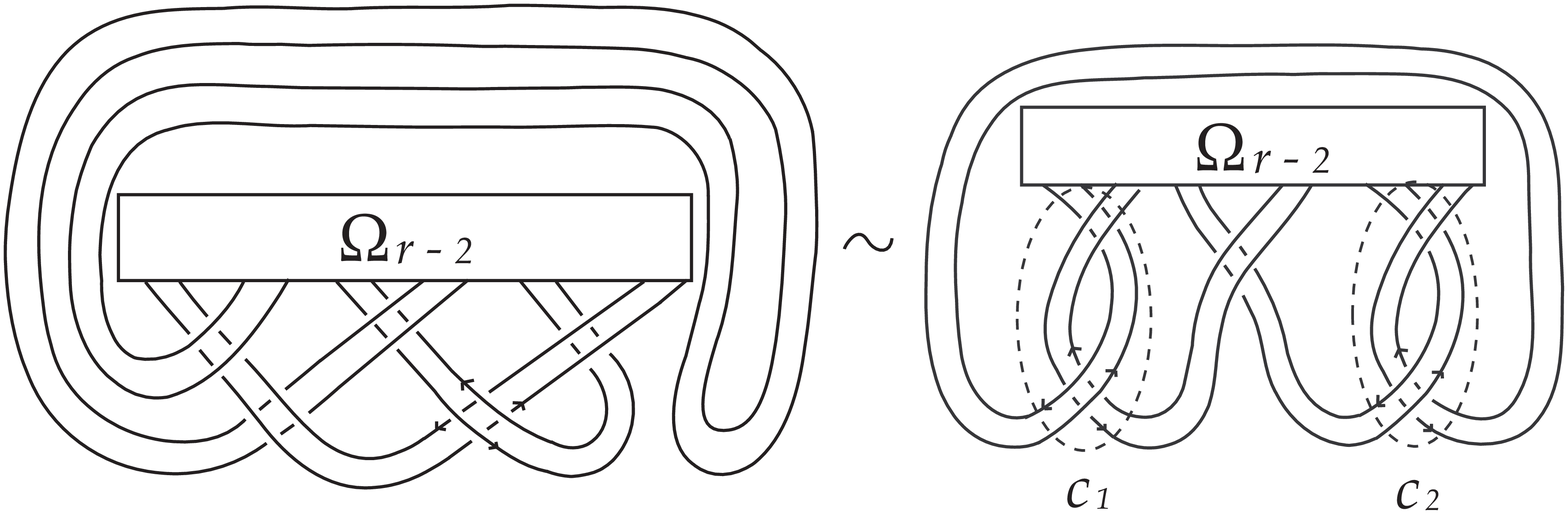}}
\caption{$D^{2}_{r}$}\label{fig-D-2-r}
\end{center}
\end{figure}

Let $L''$ be an oriented link having diagram $D''$ obtained from the standard closed braid diagram of a non-split alternating link $\hat\beta_{r-1}$ by replacing two crossings $\sigma_1^{\epsilon_{r-11}}$ and $\sigma_1^{\epsilon_{r-13}}$ in $\hat\beta_{r-1}$ with full twists, respectively, as illustrated in Figure \ref{fig-D''-pf-main-thm-1}.
 \begin{figure}[ht]
\begin{center}
\resizebox{0.60\textwidth}{!}{%
  \includegraphics{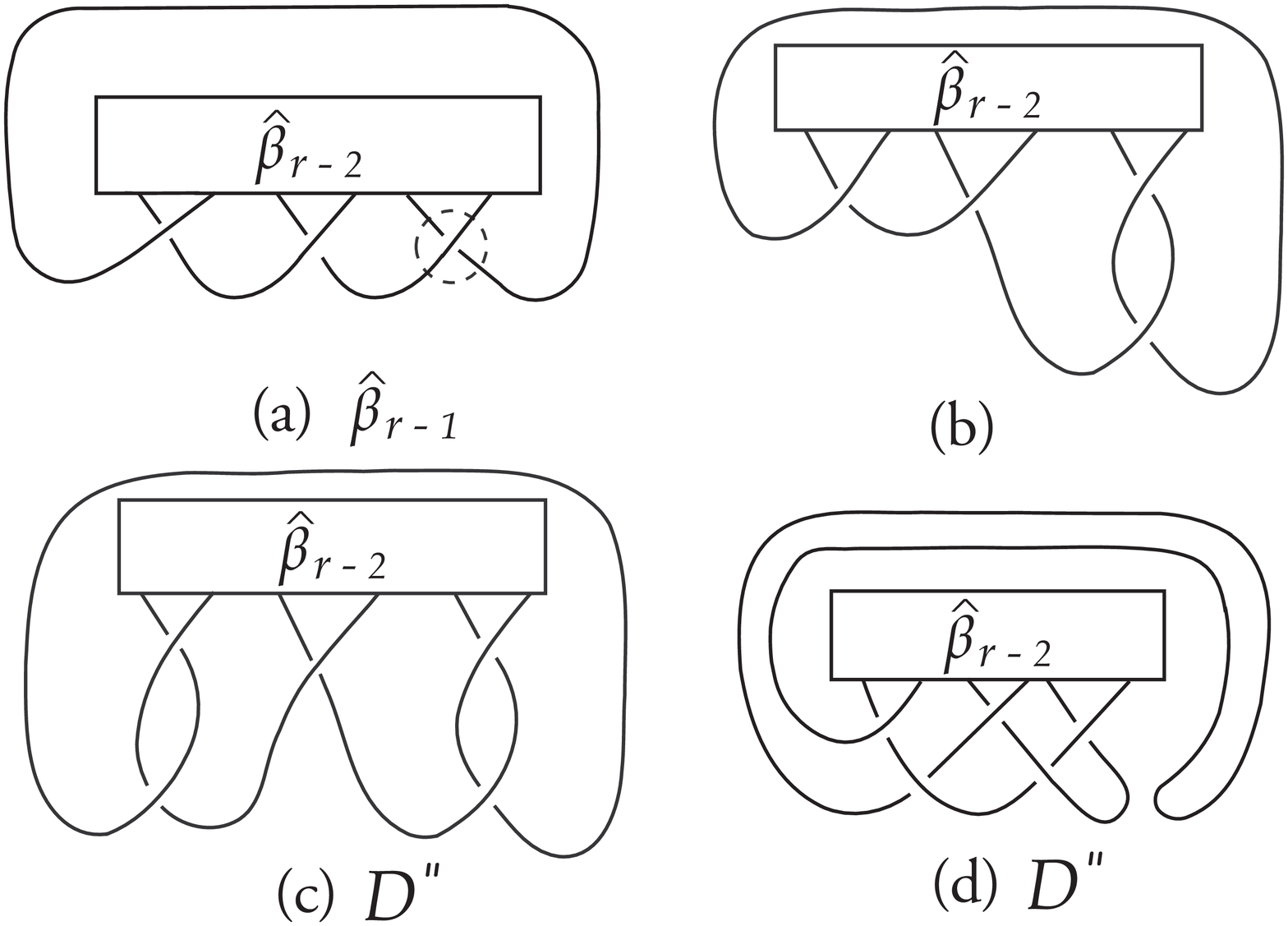}}
\caption{$D''$}\label{fig-D''-pf-main-thm-1}
\end{center}
\end{figure}
So $c(D'')=c(\hat\beta_{r-1})+2$. It is obvious that the doubled link $W_2(L'')$ has a diagram in the right side of Figure \ref{fig-D-2-r}. By induction hypothesis and Proposition \ref{BJprop4-cr-nbr-cg-wd}, we then have
\begin{align}
\max\deg_z P_{D^2_r}(v,z) &=\max\deg_z P_{W_2(L'')}(v,z)\notag\\&=2c(D'')-1\notag\\&=\max\deg_z P_{W_2(\hat\beta_{r-1})}(v,z)+4\notag\\
&=\max\deg_z P_{D_{r-1}}(v,z)+4.\label{eqn5-pf-main-thm-1}
\end{align}
Since $\max\deg_z P_{D_{r}^3}(v,z)$ is too low to interfere with our main calculation by applying
Morton's inequality, we see that maximal degree in $z$ for $P_{D_{r}^3}(v,z)$ does not contribute anything to $\max\deg_z P_{D_{r}}(v,z)$. From (\ref{eqn1-pf-main-thm-1}), (\ref{eqn4-pf-main-thm-1}), (\ref{eqn5-pf-main-thm-1}) and Lemma \ref{main-lem-1}, it is easily seen that
\begin{equation}\label{eqn8-pf-main-thm-1}
\max\deg_z P_{D_r}(v,z) = \max\{M(D_{r-1})+6, N_r-2\}.
\end{equation}
On the other hand, we see from (\ref{defn-N_r}) and (\ref{main-formula-1-r-1}) that
\begin{align}\label{eqns-pf-main-thm-1}
M(D_{r-1})+6&=\max\deg_z P_{D_{r-1}}(v,z) + 6\notag\\
&=\max\deg_z P_{W_2(\hat\beta_{r-1})}(v,z)+6\notag\\
&= (2c(\hat\beta_{r-1})-1)+6\notag\\
&=6r-1\notag\\
&=N_r~ (r \geq 2).
\end{align}
Hence it follows from (\ref{eqn8-pf-main-thm-1}) and (\ref{eqns-pf-main-thm-1}) that
\begin{equation}\label{eqn6-pf-main-thm-1}
\max\deg_z P_{D_r}(v,z) = N_r =\max\deg_z P_{D_{r-1}}(v,z)+6.\end{equation}
Combining (\ref{main-formula-1-r-1}) and (\ref{eqn6-pf-main-thm-1}), we finally obtain
\begin{align*}
\max\deg_z P_{D_r}(v,z)
&=\max\deg_z P_{D_{r-1}}(v,z)+6\\
&=2c(\hat\beta_{r-1})-1 + 6\\
&=2(c(\hat\beta_{r-1})+3)-1\\
&=2c(\hat\beta_{r})-1.
\end{align*}

{\bf Case II.} $\epsilon_{r3}=-1$.

\bigskip

In this case, it follows from the condition (\ref{braid-beta-1}) that $\epsilon_{r1}=\epsilon_{r2}=-1$. Then it is easily seen that the corresponding link diagram $W_2(\hat\beta_r)$ is just the mirror image of the diagram $D_r$ for which the assertion has already been established in the previous case I. On the other hand, it is well known that if $L^*$ is the mirror image of an oriented link $L$, then  $P_{L^*}(v,z)=P_{L}(v^{-1},z)$. This fact implies that
$P_{W_2(\hat\beta_r)}(v,z)=P_{D_r}(v^{-1},z)$. Hence
\begin{align*}
\max\deg_z P_{W_2(\hat\beta_r)}(v,z)&=\max\deg_z P_{D_r}(v^{-1},z)\\
&=\max\deg_z P_{D_r}(v,z)\\&=2c(\hat\beta_r)-1.
\end{align*} Finally, it is straightforward from (\ref{cr-nbr-beta-r}) that $2c(\hat\beta_{r})-1 =6r-1$ for each $r \geq 1$. This completes the proof of Theorem \ref{main-thm-1}.
\end{proof}


\section{A family of alternating knots for which Tripp's conjecture holds}\label{sect-fkmis}

Let us begin this section with the following:

\begin{lem}\label{main-lem-2}
Let $\beta_r (r \geq 1)$ be a quasitoric braid of type $(r+1,3)$ in (\ref{braid-beta}). If $L$ is a link having diagram $D$ obtained from the standard closed braid diagram of $\hat\beta_r$ as shown in Figure \ref{fig5.1} by replacing a crossing with a full twist (so that $c(D)=c(\hat\beta_r)+1$), then $$\max\deg_z P_{W_2(D)}(v,z)=2c(D)-1.$$
\end{lem}

\begin{proof}
Let $L'$ be the link represented by a quasitoric braid $\beta_r$. It is obvious that $L'$ is a non-split alternating link with a diagram $D'=\hat\beta_r$ satisfying $c(L')=c(D')=3r$. By Theorem \ref{main-thm-1}, $\max\deg_z P_{W_2(D')}(v,z)=2c(D')-1$. Hence the assertion follows from Proposition \ref{BJprop4-cr-nbr-cg-wd}.
 \end{proof}

\begin{thm}\label{main-thm-2}
Let $\beta_r (r \geq 1)$ be a quasitoric braid of type $(r+1,3)$ in (\ref{braid-beta}) and let $\mathcal K_r$ be the class consisting of the alternating knot $\hat\beta_r$ itself (if it is a knot) and all alternating knots having diagrams which can be obtained from the standard diagram of the closed braid $\hat\beta_r$ as shown in Figure \ref{fig5.1}, by repeatedly replacing a crossing by a full twist. Then for every $K \in \mathcal K_r$ and any integer $m$, \begin{equation}\label{eqn0-thm-2}
\max\deg_z P_{W_\pm(K,m)}(v,z)=2c(K),
\end{equation}
and therefore $$g_c(W_\pm(K,m))=c(K).$$
\end{thm}

\begin{proof} Let $K$ be an alternating knot in $K_r$. Then $K$ has a diagram $D$ which is obtained from the standard diagram of the closed braid $\hat\beta_r$ by repeatedly replacing a crossing by a full twist. By Lemma \ref{main-lem-2} and repeatedly applying Proposition \ref{BJprop4-cr-nbr-cg-wd}, we obtain
\begin{equation}\label{pf-eqn1-thm-2}
\max\deg_z P_{W_2(D)}(v,z)=2c(D)-1.
\end{equation}
Now, for any given integer $m$, let $W_\pm(K,m)$ be the $m$-twisted positive/negative Whitehead double of $K$ and let $W_\pm(D,m)$ be the canonical diagram for $W_\pm(K,m)$ associated with $D$. Since $c(D) > 3$, it follows from (\ref{pf-eqn1-thm-2}) and Proposition \ref{prop3-cr-nbr-cg-wd} that $\max\deg_z P_{W_\pm(K,m)}(v,z) > 0$ and hence \\$\max\deg_z P_{W_2(D, w(D))}(v,z) \not= 1$. By (\ref{eq1-prop3-cr-nbr-cg-wd}) and (\ref{eq2-prop3-cr-nbr-cg-wd}), we have
\begin{align*}
 \max\deg_z P_{W_\pm(K,m)}(v,z)
 &=\max\deg_z P_{W_\pm(D,m)}(v,z)\\
 &=\max\deg_z P_{W_2(D,m)}(v,z) +1\\
 &=\max\deg_z P_{W_2(D,w(D))}(v,z) +1\\
 &=\max\deg_z P_{W_2(D)}(v,z)+1\\
 &=2c(D)-1+1\\
 &=2c(D)=2c(K).
\end{align*}
This establishes the desired identity (\ref{eqn0-thm-2}).

Finally, it follows from (\ref{Morton-ineq-c-gen-2}) and (\ref{eqn0-thm-2}) that
\begin{align*}
c(K)&=\frac{1}{2} \max\deg_z P_{W_\pm(K,m)}(v,z) \leq g_c(W_\pm(K,m))\notag\\ &\leq g_c(W_\pm(D,m))=c(K).
\end{align*}
This gives $g_c(W_\pm(K,m))=c(K)$ and competes the proof.
\end{proof}

\begin{rem} 
(1) The closure $\hat\beta_1$ of the quasitoric braid  $\beta_1=(\sigma^{\epsilon_{11}})^3$ is the right-handed trefoil or left-handed trefoil knot (see Remark \ref{prop-r-qtb} (1)) and so the class $\mathcal K_1$ in Theorem \ref{main-thm-2} is just the class $\mathcal K$ in Proposition \ref{BJprop3-cr-nbr-cg-wd}. So, in case of $r=1$, Theorem \ref{main-thm-2} is the same as Proposition \ref{BJprop3-cr-nbr-cg-wd}. Hence $\mathcal K_1$ contains all $(2, n)$-torus knots, all the $2$-bridge knots, and all alternating  pretzel knots. 

(2) In \cite{BJ}, Brittenham and Jensen noticed that the Borromean ring $L$, the closure of the quasitoric braid $\beta_2$, satisfy $\max\deg_z P_{W_2(L)}(v,z)=2c(L)-1$ (see Example \ref{examp1-br}), which give rise, using Proposition \ref{BJprop4-cr-nbr-cg-wd}, to a family, it is indeed the family $K_2$ in Theorem \ref{main-thm-2}, of alternating knots satisfying the equality (\ref{eq-tripp-conj}), different from the family $\mathcal K$ given by Proposition \ref{BJprop3-cr-nbr-cg-wd}. On the other hand, it is clear that $\hat\beta_2 \notin \mathcal K_3$ and so $\mathcal K_3$ is also a family of alternating knots satisfying the equality (\ref{eq-tripp-conj}), different from $\mathcal K_2$, and so on. Therefore, Theorem \ref{main-thm-2} provides an infinite sequence $$\mathcal K_1(=\mathcal K), \mathcal K_2, \mathcal K_3, \ldots, \mathcal K_i, \ldots$$ of infinite families $\mathcal K_i$ of alternating knots satisfying Tripp-Nakamura's Conjecture. We define
$$\mathcal K^3=\bigcup_{r=1}^\infty\mathcal K_r.$$ Then the infinite family $\mathcal K^3$ of alternating knots is an extension of the previous results of Tripp \cite{Tri}, Nakamura \cite{Nak} and Brittenham-Jensen \cite{BJ}.
\end{rem}

\begin{examp}
Let $A=(n_{ij})_{1\leq i \leq r; 1 \leq j \leq 3}$ be an arbitrary given $r\times 3$ integral matrix, i.e.,
$$A=\left(
\begin{array}{ccc}
n_{11}&n_{12}& n_{13}\\
n_{21}&n_{22}& n_{23}\\
\vdots & \vdots & \vdots\\
n_{r1}&n_{r2}& n_{r3}
\end{array}
\right).$$ 
Let $K_A$ denote an oriented link in $S^3$ having a diagram $D_A$ as shown in Figure \ref{fig6.1} (a) in which each tangle labeled a non-zero integer $n_{ij}$ denotes a vertical $n_{ij}$ half-twists as shown in Figure~\ref{fig6.1}(b) or a horizontal $n_{ij}$ half-twists.
\begin{figure}[tb]
\vspace*{10pt} \centerline{ \xy
 (85,0);(90,0) **@{-}, (85,6);(90,6) **@{-},
 (85,0);(85,6) **@{-}, (90,0);(90,6) **@{-}, (87.5,3)*{_{n_{r2}}},
 (82,12);(87,12) **@{-}, (82,18);(87,18) **@{-},
 (82,12);(82,18) **@{-}, (87,12);(87,18) **@{-}, (84.5,15)*{_{n_{32}}},
 (75,16);(80,16) **@{-}, (75,22);(80,22) **@{-},
 (75,16);(75,22) **@{-}, (80,16);(80,22) **@{-}, (77.5,19)*{_{n_{22}}},
 (68,20);(73,20) **@{-}, (68,26);(73,26) **@{-},
 (68,20);(68,26) **@{-}, (73,20);(73,26) **@{-}, (70.5,23)*{_{n_{12}}},
 (66,1);(85,1) **@{-}, (90,1);(92,1) **@{-},
 (83,5);(85,5) **@{-}, (90,5);(92,5) **@{-},
 (80,13);(82,13) **@{-}, (87,13);(89,13) **@{-},
 (66,17);(75,17) **@{-}, (80,17);(82,17) **@{-}, (87,17);(92,17) **@{-},
 (66,21);(68,21) **@{-}, (73,21);(75,21) **@{-}, (80,21);(92,21) **@{-},
 (66,25);(68,25) **@{-}, (73,25);(92,25) **@{-}, (86,10)*{\ddots},
 (59,0);(64,0) **@{-}, (59,6);(64,6) **@{-},
 (59,0);(59,6) **@{-}, (64,0);(64,6) **@{-}, (61.5,3)*{_{n_{r1}}},
 (56,12);(61,12) **@{-}, (56,18);(61,18) **@{-},
 (56,12);(56,18) **@{-}, (61,12);(61,18) **@{-}, (58.5,15)*{_{n_{31}}},
 (49,16);(54,16) **@{-}, (49,22);(54,22) **@{-},
 (49,16);(49,22) **@{-}, (54,16);(54,22) **@{-}, (51.5,19)*{_{n_{21}}},
 (42,20);(47,20) **@{-}, (42,26);(47,26) **@{-},
 (42,20);(42,26) **@{-}, (47,20);(47,26) **@{-}, (44.5,23)*{_{n_{11}}},
 (40,1);(59,1) **@{-}, (64,1);(66,1) **@{-},
 (57,5);(59,5) **@{-}, (64,5);(66,5) **@{-},
 (54,13);(56,13) **@{-}, (61,13);(63,13) **@{-},
 (40,17);(49,17) **@{-}, (54,17);(56,17) **@{-}, (61,17);(66,17) **@{-},
 (40,21);(42,21) **@{-}, (47,21);(49,21) **@{-}, (54,21);(66,21) **@{-},
 (40,25);(42,25) **@{-}, (47,25);(66,25) **@{-}, (60,10)*{\ddots},
 (125,0);(130,0) **@{-}, (125,6);(130,6) **@{-},
 (125,0);(125,6) **@{-}, (130,0);(130,6) **@{-}, (127.5,3)*{_{n_{r3}}},
 (122,12);(127,12) **@{-}, (122,18);(127,18) **@{-},
 (122,12);(122,18) **@{-}, (127,12);(127,18) **@{-}, (124.5,15)*{_{n_{33}}},
 (115,16);(120,16) **@{-}, (115,22);(120,22) **@{-},
 (115,16);(115,22) **@{-}, (120,16);(120,22) **@{-}, (117.5,19)*{_{n_{23}}},
 (108,20);(113,20) **@{-}, (108,26);(113,26) **@{-},
 (108,20);(108,26) **@{-}, (113,20);(113,26) **@{-}, (110.5,23)*{_{n_{13}}},
 (106,1);(125,1) **@{-}, (130,1);(132,1) **@{-},
 (123,5);(125,5) **@{-}, (130,5);(132,5) **@{-},
 (120,13);(122,13) **@{-}, (127,13);(129,13) **@{-},
 (106,17);(115,17) **@{-}, (120,17);(122,17) **@{-}, (127,17);(132,17) **@{-},
 (106,21);(108,21) **@{-}, (113,21);(115,21) **@{-}, (120,21);(132,21) **@{-},
 (106,25);(108,25) **@{-}, (113,25);(132,25) **@{-}, (126,10)*{\ddots},
 (40,31);(132,31) **@{-}, (40,33);(132,33) **@{-}, (40,35);(132,35)
 **@{-}, (40,42);(132,42) **@{-}, (50,39)*{\vdots},
 (122,39)*{\vdots},
 (92,1);(108,1) **@{-},
 (92,17);(108,17) **@{-},
 (92,21);(108,21) **@{-},
 (92,25);(108,25) **@{-},
 (40,25);(40,31) **\crv{(38,25)&(38,31)}, (40,21);(40,33)
 **\crv{(36,21)&(36,33)}, (40,17);(40,35) **\crv{(34,17)&(34,35)},
 (40,1);(40,42) **\crv{(30,1)&(30,42)},
 (132,25);(132,31) **\crv{(134,25)&(134,31)}, (132,21);(132,33)
 **\crv{(136,21)&(136,33)}, (132,17);(132,35) **\crv{(138,17)&(138,35)},
 (132,1);(132,42) **\crv{(142,1)&(142,42)},
 (77,25);(78,25) **@{} ?>*\dir{>}, (50,25);(52,25) **@{} ?>*\dir{>},
 (117,25);(118,25) **@{} ?>*\dir{>}, (87.5,-7)*{(a) ~D_A},
 \endxy}
 \vskip 0.5cm
 \centerline{\xy (52,19);(60,19) **@{-},
(52,27);(60,27) **@{-}, (52,19);(52,27) **@{-}, (60,19);(60,27)
**@{-}, (50,20);(52,20) **@{-}, (50,26);(52,26) **@{-},
(60,20);(62,20) **@{-}, (60,26);(62,26) **@{-}, (56,23)*{n_{ij}},
(70,23)*{=},
 (80,41);(82.5,37) **\crv{(80,40)&(81,37.5)}, (85,33);(82.5,37) **\crv{(85,34)&(84,36.5)},
 (85,41);(83.3,37.5) **\crv{(85,40)&(84,38)}, (80,33);(81.7,36.5) **\crv{(80,34)&(81,36)},
 (80,33);(82.5,29) **\crv{(80,32)&(81,29.5)}, (85,25);(82.5,29) **\crv{(85,26)&(84,28.5)},
 (85,33);(83.3,29.5) **\crv{(85,32)&(84,30)}, (80,25);(81.7,28.5) **\crv{(80,26)&(81,28)},
 (80,21);(82.5,17) **\crv{(80,20)&(81,17.5)}, (85,13);(82.5,17) **\crv{(85,14)&(84,16.5)},
 (85,21);(83.3,17.5) **\crv{(85,20)&(84,18)}, (80,13);(81.7,16.5) **\crv{(80,14)&(81,16)},
 (82.5,24)*{\vdots}, (82.5,7)*{n_{ij} > 0},
 (90,23)*{,},
 (100,41);(97.5,37) **\crv{(100,40)&(99,37.5)}, (95,33);(97.5,37) **\crv{(95,34)&(96,36.5)},
 (95,41);(96.7,37.5) **\crv{(95,40)&(96,38)}, (100,33);(98.3,36.5) **\crv{(100,34)&(99,36)},
 (100,33);(97.5,29) **\crv{(100,32)&(99,29.5)}, (95,25);(97.5,29) **\crv{(95,26)&(96,28.5)},
 (95,33);(96.7,29.5) **\crv{(95,32)&(96,30)}, (100,25);(98.3,28.5) **\crv{(100,26)&(99,28)},
 (100,21);(97.5,17) **\crv{(100,20)&(99,17.5)}, (95,13);(97.5,17) **\crv{(95,14)&(96,16.5)},
 (95,21);(96.7,17.5) **\crv{(95,20)&(96,18)}, (100,13);(98.3,16.5) **\crv{(100,14)&(99,16)},
 (97.5,24)*{\vdots}, (97.5,7)*{n_{ij} < 0}, (77,-3)*{(b)},
\endxy
 }
\vspace*{10pt}\caption{}\label{fig6.1}
\end{figure}
Suppose that $n_{ij}n_{i+1j} < 0$ and $n_{ij}n_{ij+1} > 0$ for each $i=1,2,\ldots, r-1$ and $j=1,2,3$ and $K_A$ is a knot (eventually, an alternating knot). Let $A'=(\epsilon_{ij})_{1\leq i \leq r; 1 \leq j \leq 3}$ be the integral matrix obtained from $A$ by defining $\epsilon_{ij}=\frac{n_{ij}}{|n_{ij}|}~ (1\leq i \leq r; 1 \leq j \leq 3)$ and let $K_{A'}$ be the oriented alternating link having a diagram $D_{A'}$. Then $K_{A'}$ is the closure of a quasitoric braid $\beta_r$ in (\ref{braid-beta}). Then it follows from  Theorem \ref{main-thm-2} that $K_A \in \mathcal K_r$ and so
\begin{equation*}
\max\deg_z P_{W_2(K_A)}(v,z)=\max\deg_z P_{W_2(K_{A'})}(v,z)+2\sum_{i=1}^r\sum_{j=1}^3(|n_{ij}|-1).
\end{equation*}
Consequently. for every integer $m$,
$$g_c(W_\pm(K_A,m))=\sum_{i=1}^r\sum_{j=1}^3|n_{ij}|=c(K_A).$$
\end{examp}


\section{Proof of Lemma \ref{main-lem-1}}
\label{sect-pf-lemma}

In this section, we prove Lemma \ref{main-lem-1}. For this purpose, we first remind that $D_r$ denotes the doubled link $W_2(\hat\beta_r)$ corresponding to the matrix notation $Q_r$ with $\epsilon_{r3}=1$. We also remind that $D_{r}^i$ ($4 \leq i \leq 8$) denotes the link diagram  obtained from $D_r$ by replacing $T^{1}_{r3}$ with $T_i$, where $T_4=F_1, T_5=F_2, T_6=F_3, T_7=F_4, T_8=G$ (cf. Section \ref{sect-miwd3-plrq}).

\bigskip

\noindent{\bf Proof of (1).}  Consider a partial skein tree for $D_{r}^{4}~ (r\geq 3)$ and isotopy deformations as shown in Figure \ref{res-tree-dr41}, which yields the identity:
\begin{figure}[ht]
\begin{center}
\resizebox{0.65\textwidth}{!}{%
  \includegraphics{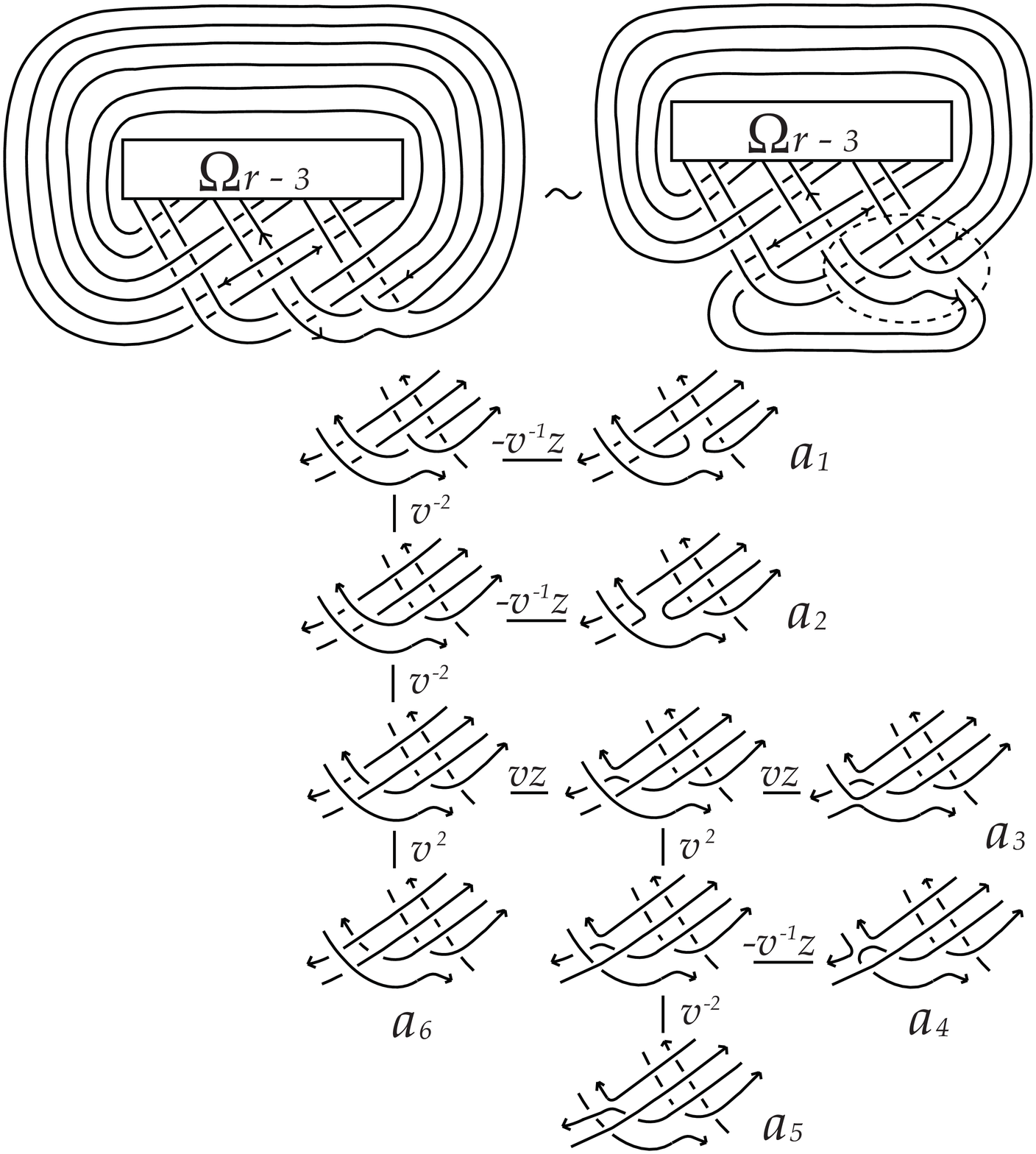}}
\caption{A partial skein tree for $D_{r}^{4}$.}\label{res-tree-dr41}
\end{center}
\end{figure}
\begin{align}\label{eq3-pf-(1)-lem-0}
P_{D_{r}^{4}}(v,z)&=v^{-2}P_{a_6}(v,z)+v^{-3}zP_{a_5}(v,z)-v^{-2}z^2P_{a_4}(v,z)\notag\\
&+v^{-2}z^2P_{a_3}(v,z)-v^{-3}zP_{a_2}(v,z)-v^{-1}zP_{a_1}(v,z).
\end{align}
It is clear from Figure \ref{res-tree-dr41} that the link $a_{1}$ does not contribute anything to $\max\deg_zP_{D_{r}^{4}}(v,z)$. For the links $a_2, a_4$ and $a_5$, it follows from Morton's inequality in (\ref{Morton-ineq-c-gen}) that
\begin{align}
\max\deg_zP_{a_2}(v,z)
&\leq c(a_{2})-s(a_{2})+1\notag\\
&\leq (c(D_r)-6)-(s(D_r)-2)+1\notag\\
&=N_r-4,\label{eq3-pf-(1)-lem-1}\\
\max\deg_zP_{a_{4}}(v,z)
&\leq c(a_{4})-s(a_{4})+1\notag\\
&\leq (c(D_r)-7)-(s(D_r)-2)+1\notag\\
&=N_r-5,\label{eq3-pf-(1)-lem-2}\\
\max\deg_zP_{a_{5}}(v,z)
&\leq c(a_{5})-s(a_{5})+1\notag\\
&\leq (c(D_r)-11)-(s(D_r)-5)+1\notag\\
&=N_r-6.\label{eq3-pf-(1)-lem-3}
\end{align}
For the link $a_{3}$, we obtain from Figure \ref{res-tree-dr42} that
$$P_{a_{3}}(v,z)=v^2P_{a_{8}}(v,z)+vzP_{a_{7}}(v,z).$$

\begin{figure}[ht]
\begin{center}
\resizebox{0.55\textwidth}{!}{%
  \includegraphics{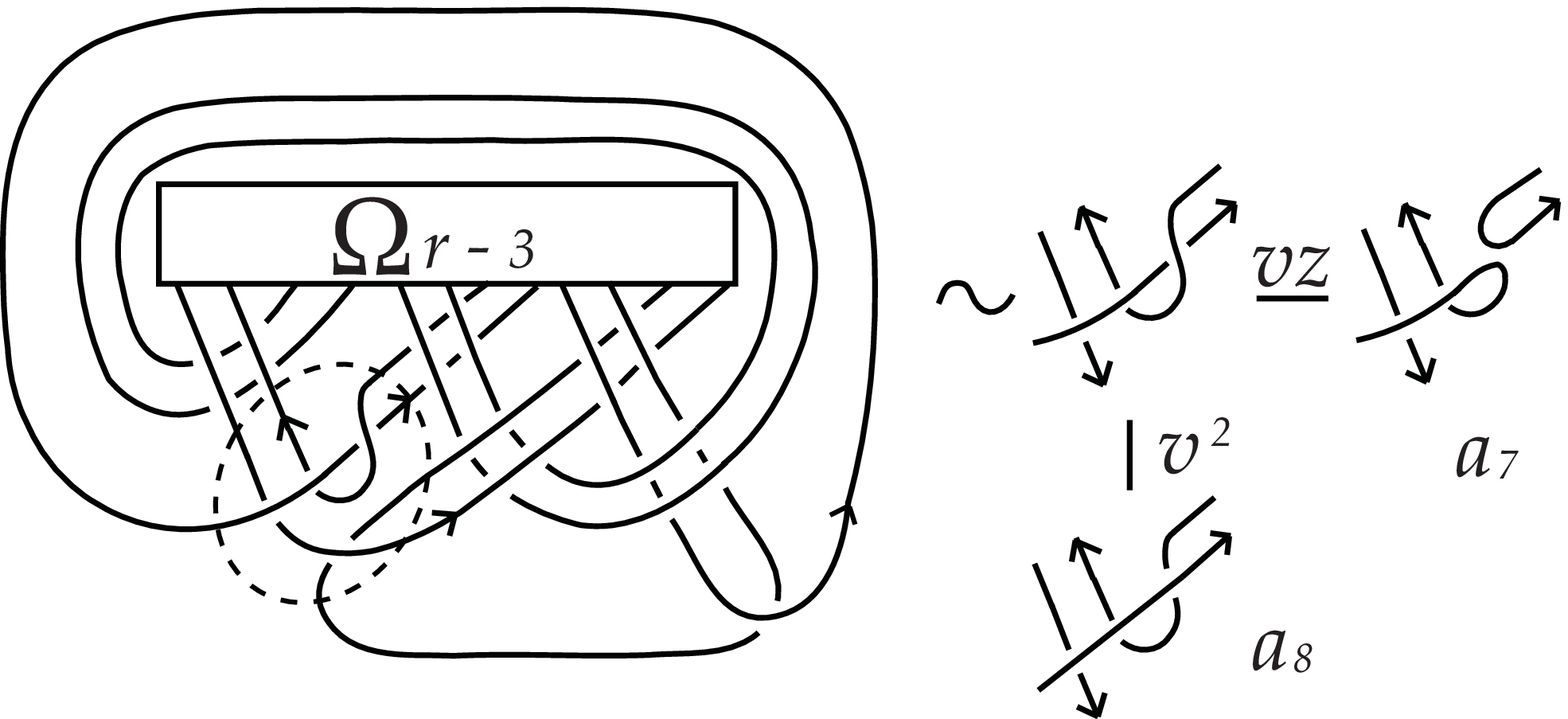}}
\caption{A partial skein tree for $a_3$.}\label{res-tree-dr42}
\end{center}
\end{figure}

Clearly, the link $a_{7}$ does not contribute anything to $\max\deg_zP_{a_3}(v,z)$ and so by Morton's inequality,
\begin{align}
\max\deg_zP_{a_{3}}(v,z)&=\max\deg_zP_{a_{8}}(v,z)
\leq c(a_{8})-s(a_{8})+1\notag\\
&\leq (c(D_r)-13)-(s(D_r)-6)+1\notag\\
&=N_r-7.\label{eq3-pf-(1)-lem-4}
\end{align}

From (\ref{eq3-pf-(1)-lem-0})-(\ref{eq3-pf-(1)-lem-4})  and Claim 1 below, we obtain
\begin{align*}
\max&\deg_zP_{D_{r}^{4}}(v,z)\\
&\leq \max\{M(a_6), M(a_5)+1, M(a_{4})+2,  M(a_{3})+2,  M(a_{2})+1\}\\
&\leq \max\{M(a_6), N_r-5, N_r-3, N_r-5, N_r-3\}\\
&= N_r-3.
\end{align*}
This establishes (1), as desired.

\bigskip

{\bf Claim 1.} $M(a_6)=\max\deg_zP_{a_6}(v,z) \leq N_r-3 ~(r \geq 3).$

\bigskip

{\bf Proof of Claim 1.} Consider a partial skein tree for $a_6$ and isotopy deformations as shown in Figure \ref{res-tree-dr43}, which gives the identity:
\begin{align}\label{eq3-pf-(1)-lem-1}
P_{a_6}(v,z)&=P_{a_{12}}(v,z)-v^{-1}zP_{a_{11}}(v,z)+z^{2}P_{a_{10}}(v,z)+vzP_{a_{9}}(v,z).
\end{align}

\begin{figure}[ht]
\begin{center}
\resizebox{0.70\textwidth}{!}{%
  \includegraphics{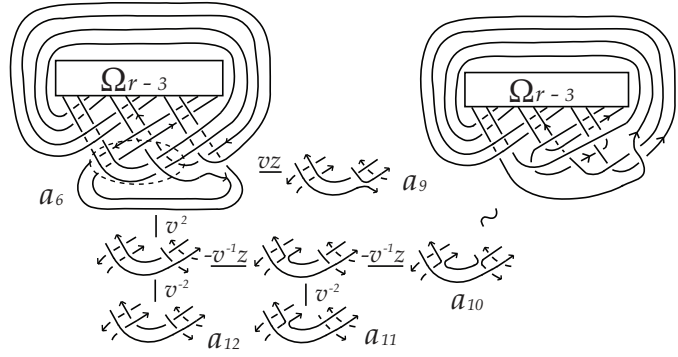}}
\caption{A partial skein tree for $a_6$.}\label{res-tree-dr43}
\end{center}
\end{figure}

Using Morton's inequality, we obtain
\begin{align}
\max\deg_zP_{a_{12}}(v,z)
&\leq (c(D_r)-6)-(s(D_r)-3)+1
=N_r-3,\label{eq3-pf-(1)-lem-5}\\
\max\deg_zP_{a_{11}}(v,z)
&\leq (c(D_r)-8)-(s(D_r)-4)+1
=N_r-4,\label{eq3-pf-(1)-lem-6}\\
\max\deg_zP_{a_{10}}(v,z)
&\leq (c(D_r)-8)-(s(D_r)-3)+1
=N_r-5\label{eq3-pf-(1)-lem-7}.
\end{align}

\begin{figure}[ht]
\begin{center}
\resizebox{0.75\textwidth}{!}{%
  \includegraphics{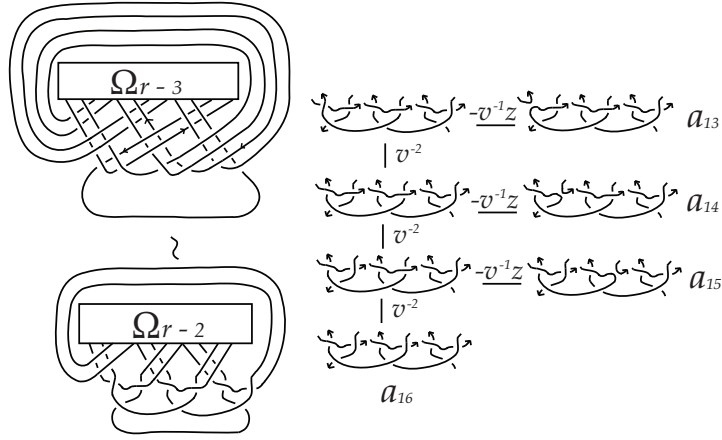}}
\caption{A partial skein tree for $a_{9}$.}\label{res-tree-dr45}
\end{center}
\end{figure}

By a partial skein tree for $a_{9}$ and isotopy deformations as shown in Figure \ref{res-tree-dr45}, we get
\begin{align*}
P_{a_{9}}(v,z)
&=v^{-6}P_{a_{16}}(v,z)-v^{-5}zP_{a_{15}}(v,z)-\notag\\&~~~~~~~~~v^{-3}zP_{a_{14}}(v,z)-v^{-1}zP_{a_{13}}(v,z).
\end{align*}

It is clear that the links $a_{13}$, $a_{14}$ and $a_{15}$ do not contribute anything to $\max\deg_zP_{a_{9}}(v,z)$. Then
\begin{equation}\label{eq3-pf-(1)-lem-18}
\max\deg_zP_{a_{9}}(v,z)=\max\deg_zP_{a_{16}}(v,z).
\end{equation}

In the link diagram $a_{16}$, we consider the three crossings labeled 1, 2 and 3 in the $(r-1)$-th row as indicated in the first row of Figure \ref{res-tree-dr46} according as the case (a) $r\equiv 2$ (mod 3), (b) $r\equiv 0$ (mod 3) and (c) $r\equiv 1$ (mod 3). 
\begin{figure}[ht]
\begin{center}
\resizebox{0.90\textwidth}{!}{%
  \includegraphics{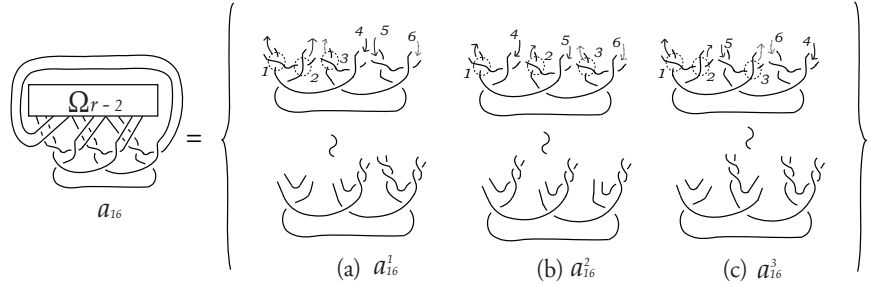}}
\caption{$r\equiv2$, ~$r\equiv0$,~$r\equiv1$ (mod 3).}\label{res-tree-dr46}
\end{center}
\end{figure}
\begin{figure}[ht]
\begin{center}
\resizebox{0.75\textwidth}{!}{%
  \includegraphics{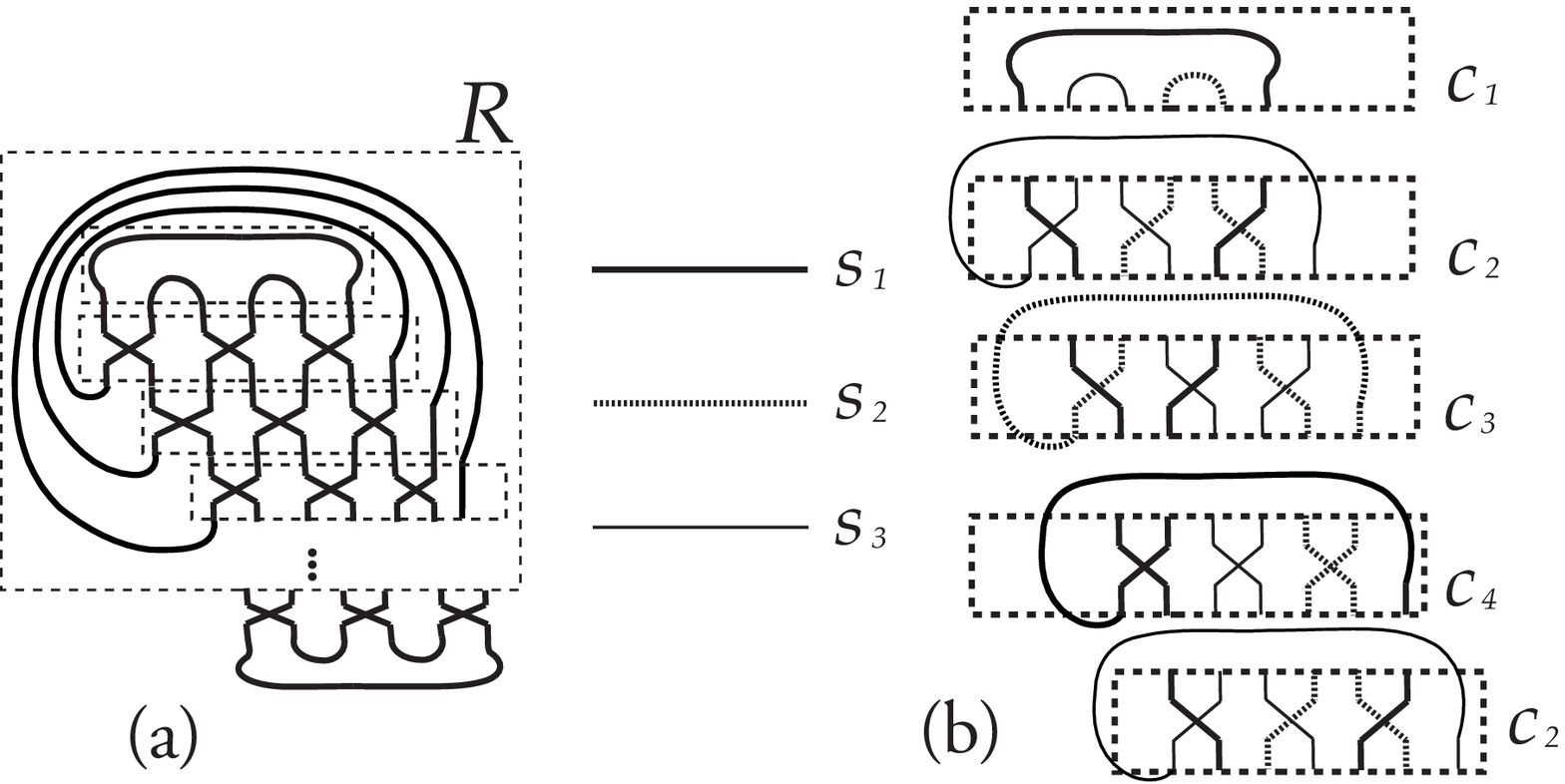}}
\caption{}\label{dr8}
\end{center}
\end{figure}
For a regular projection of $\hat\beta_r$ as shown in Figure \ref{dr8} $(a)$, we observe that there are three arcs, say $S_1, S_2, S_3$, in the dotted rectangle $R$ in Figure \ref{dr8} $(a)$ that are obtained from the arcs in the small dotted rectangles $C_1, C_2, C_3, C_4$ in $R$ as shown in Figure \ref{dr8} $(b)$ by gluing them in the obvious way, written $R=C_1C_2C_3C_4.$ From this, it is not difficult to see in general that \begin{equation}\label{cr-sliding}
\hat\beta_r=C_1C_2C_3C_4C_2C_3C_4\cdots C_m,
\end{equation} where \begin{align*}
 &m=2, ~r\equiv 2~(mod~3),\notag\\
 &m=3,~r\equiv 0~(mod~3),\notag\\
 &m=4, ~r\equiv 1~(mod~3).
\end{align*}
Pushing each crossing labeled 1, 2, 3 into the part of $\Omega_{r-2}$ along the 2-parallel strings, it follows from (\ref{cr-sliding}) that it returns to the arrow labeled 4, 5, 6 in the $(r-1)$-th row, respectively, illustrated in (a), (b) and (c) of Figure \ref{res-tree-dr46} according as the case $r\equiv 2$ (mod 3), $r\equiv 0$ (mod 3) and $r\equiv 1$ (mod 3). 

Now, by a similar argument in the proof of Proposition \ref{prop3-cr-nbr-cg-wd} (2), the full twists in $a_{16}^{i}$ can be removed from without contributing to $\max\deg_zP_{a_{16}^{i}}(v,z)$ for each $i=1,2,3$ and so we obtain
$$\max\deg_zP_{a_{16}^{i}}(v,z)=\max\deg_zP_{a_{17}}(v,z),$$ where $a_{17}$ is the link diagram as shown in Figure \ref{res-tree-dr4-1}.

\begin{figure}[ht]
\begin{center}
\resizebox{0.30\textwidth}{!}{%
  \includegraphics{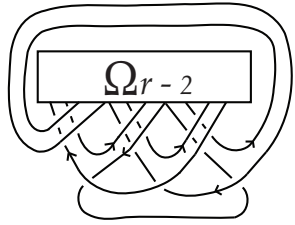}}
\caption{A partial skein tree for $a_{17}$.}\label{res-tree-dr4-1}
\end{center}
\end{figure}

On the other hand, by Morton's inequality, we obtain
\begin{align}
\max\deg_zP_{a_{17}}(v,z)
&\leq (c(D_r)-9)-(s(D_r)-5)+1
=N_r-4.\label{eq3-pf-(1)-lem-19}
\end{align}
Then it is direct from (\ref{eq3-pf-(1)-lem-18}) and (\ref{eq3-pf-(1)-lem-19}) that
\begin{equation}\label{eq3-pf-(1)-lem-20}
\max\deg_zP_{a_{9}}(v,z) \leq N_r-4.
\end{equation}
Therefore
we have from (\ref{eq3-pf-(1)-lem-1})-(\ref{eq3-pf-(1)-lem-7}) and (\ref{eq3-pf-(1)-lem-20}) that
\begin{align*}
\max\deg_zP_{a_{6}}(v,z) &\leq
\max\{M(a_{12}), M(a_{11})+1, M(a_{10}))+2, M(a_{9})+1\}\notag\\
&\leq \max\{N_r-3, N_r-3, N_r-3, N_r-3)\}
= N_r-3.
\end{align*}
This completes the proof of Claim 1. \qed

\bigskip

\noindent{\bf Proof of (2).} From a partial skein tree for $D^5_r$ as shown in Figure \ref{res-tree-dr5}, we obtain
\begin{equation*}
P_{D_r^5}(v,z)=v^2P_{b_{2}}(v,z)+vzP_{b_1}(v,z).
\end{equation*}
\begin{figure}[ht]
\begin{center}
\resizebox{0.50\textwidth}{!}{%
  \includegraphics{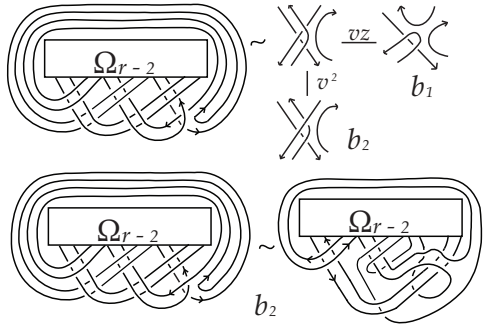}}
\caption{A partial skein tree for $D_{r}^{5}$.}\label{res-tree-dr5}
\end{center}
\end{figure}
It is quite easy to see that the link $b_1$ does not contribute anything to $\max\deg_zP_{D_{r}^{5}}(v,z)$.
By Morton's inequality, we obtain
\begin{align*}
\max\deg_zP_{D_{r}^{5}}(v,z)&=\max\deg_zP_{b_{2}}(v,z)\\
&\leq (c(D_r)-4)-(s(D_r)-1)+1=N_r-3.
\end{align*}
This completes the proof of (2). \qed

\bigskip

\noindent{\bf Proof of (3).} It follows from Morton's inequality that
\begin{align*}
\max\deg_zP_{D^6_r}(v,z)
&\leq (c(D_r)-5)-(s(D_r)-2)+1=N_r-3.
\end{align*}
This completes the proof of (3). \qed

\bigskip

\noindent{\bf Proof of (4).}
By Morton's inequality and isotopy deformations as shown in Figure \ref{res-tree-dr7}, we obtain
\begin{figure}[ht]
\begin{center}
\resizebox{0.50\textwidth}{!}{%
  \includegraphics{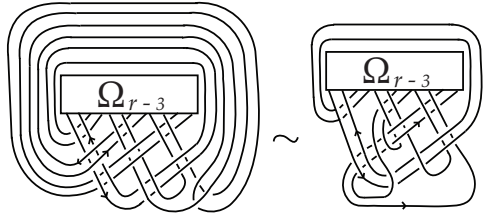}}
\caption{A partial skein tree for $D_{r}^{7}$.}\label{res-tree-dr7}
\end{center}
\end{figure}
\begin{align*}
\max\deg_zP_{D^7_r}(v,z)
&\leq (c(D_r)-4)-(s(D_r)-1)+1=N_r-3.
\end{align*}
This completes the proof of (4). \qed

\bigskip

\noindent{\bf Proof of (5).} It follows from Morton's inequality that
\begin{align*}
\max\deg_zP_{D^8_r}(v,z)
&\leq (c(D_r)-8)-(s(D_r)-4)+1=N_r-4.
\end{align*}
This completes the proof of (5). \qed


\bigskip

{\bf Acknowledgements.}
This work was supported by Basic Science Research Program through the National Research Foundation of Korea (NRF) funded by the Ministry of Education, Science and Technology (2010-0011225).


\end{document}